\documentclass[reqno,11pt]{amsart}
\usepackage{cases,comment}
\usepackage{mathrsfs}
\usepackage[T1]{fontenc}
\usepackage{mathrsfs}
\usepackage{amsmath}
\usepackage{amsmath,latexsym,amssymb,amsfonts,amsbsy, amsthm}
\usepackage{txfonts}
\usepackage{enumerate}
\usepackage{bm}
\usepackage[usenames]{color}
\usepackage{xspace,colortbl}
\usepackage{epsfig}
\usepackage{graphicx}
\usepackage{subfigure}
\usepackage{amsmath,amsfonts,amsthm,amssymb,amscd}
\usepackage{color}
\usepackage[title,titletoc,toc]{appendix}
\usepackage{bigints}

\usepackage{hyperref}

\newcommand{\be}{\begin{equation} }
\newcommand{\ee}{\end{equation}}
\newcommand{\bee}{\begin{equation*} }
\newcommand{\eee}{\end{equation*}}
\newcommand{\bse}{\begin{subequations}}
\newcommand{\ese}{\end{subequations}}

\newcommand{\R}{\mathbb{R}}

\newcommand{\rmd}{\mbox{\rm d}}
\newcommand{\imp}{\mbox{\rm Im}}

\theoremstyle{plain}
\newtheorem{theorem}{Theorem}[section]
\newtheorem{corollary}[theorem]{Corollary}

\newtheorem{lemma}{Lemma}[section]
\newtheorem{proposition}{Proposition}[section]
\theoremstyle{remark}
\newtheorem{remark}{Remark}[section]

\theoremstyle{definition}

\numberwithin{equation}{section}

\title[Stability of the BO multi-solitons ]
{Stability of multi-solitons for the Benjamin-Ono equation }

\author[Y. Lan]{Yang Lan}
\address[Yang Lan]{Yau Mathematical Sciences Center, Tsinghua University, 100084 Beijing, P. R. China}
\email{lanyang@mail.tsinghua.edu.cn}

\author[Z. Wang]{Zhong Wang}
\address[Zhong Wang]{School of Mathematics, Foshan University, 528000, P. R. China.}\email{wangzh79@fosu.edu.cn}

\date{}                                           
\begin{document}
\thispagestyle{empty}

\begin{abstract}

 This paper is concerned with the dynamical stability of the $m$-solitons of the Benjamin-Ono (BO) equation. This extends the work of Neves and Lopes \cite{LN}  which was restricted to $m=2$ the double solitons case. Multi-solitons are non-isolated constrained minimizers satisfying a suitable variational  nonlocal elliptic equation, the stability issue is reduced to the spectral analysis of higher order nonlocal operators consist of the Hilbert transform. Such operators are isoinertial and the negative eigenvalues of which can be located. Our approach in the spectral analysis consists in an invariant for the multi-solitons and new operator identities motivated by the bi-Hamiltonian structure of the BO equation. Since the BO equation is more likely a two dimensional integrable system, its recursion operator is not explicit and which contributes the main difficulties in our analysis. The key ingredient in the spectral analysis is by employing the completeness in $L^2$ of the squared eigenfunctions of the eigenvalue problem for the BO equation. It is demonstrated here that orbital  stability of soliton in $H^{\frac{1}{2}}(\R)$ implies that  all $m$-solitons are  dynamically  stable in $H^{\frac{m}{2}}(\R)$.

\end{abstract}

\maketitle

completeness relation.

\vskip 0.4cm


\setcounter{tocdepth}{1}
\section{Introduction}
\label{intro} We consider the stability of the \emph{multi-solitons} of the  Benjamin-Ono (BO) equation
\begin{equation}\tag{BO}\label{eq: BO}
u_t+Hu_{xx}+2uu_x=0, \qquad u(x,t) \in \R, \; (x,t) \in\R\times \R.
\end{equation}
Here $u=u(x,t)$ represents the amplitude of wave, and $H$ is the Hilbert transform given by
\begin{equation}\label{Hilbert}
Hu(x,t)=\frac{1}{\pi} \text{P.V.}\int_{-\infty}^{\infty}\frac{u(y,t)}{ y-x}\rmd y,
\end{equation}
where P.V. indicates that the integral is to be computed in the principle value sense.
The BO equation \eqref{eq: BO}, formulated by Benjamin \cite{benjamin1967internal} and Ono \cite{ono1975algebraic},  is used to model long internal gravity waves
in a two-layer fluid. By passing to the deep water limit, the BO equation \eqref{eq: BO} can be formally obtained from the following Intermediate Long Wave (ILW) equation (as $\delta\rightarrow+\infty$) \cite{ablowitz1991solitons},
\begin{equation}\tag{ILW}\label{eq: ILW}
u_t+\frac{1}{\delta}u_x+Tu_{xx}+2uu_x=0, \ \
(Tf)(x)=\frac{1}{2\delta}\text{P.V.}\int_{-\infty}^{\infty}\coth \frac{\pi(y-x)}{2\delta}f(y)\rmd y.
\end{equation}
whereas the shallow water limit (as $\delta\rightarrow0$) of the ILW equation gives the Korteweg-de Vries (KdV) equation
\begin{equation}\tag{KdV}\label{KdV}
u_t + \frac{\delta}{3}u_{xxx}+2uu_x=0.
\end{equation}

\eqref{eq: BO} has much in common with \eqref{KdV}. A key difference is that \eqref{eq: BO} involves a singular integro-differential operator $H$, and this leads to solitons that only have algebraic decay for \eqref{eq: BO}, as opposed to exponential decay for \eqref{KdV}.
 \eqref{eq: BO} can be written as an infinite-dimensional completely integrable Hamiltonian
dynamical system with infinitely many conservation laws and a suitable Lax-pair formulation \cite{KLM99,GK21}. In particular, the following quantities are conserved formally along the flow of \eqref{eq: BO}:
\begin{eqnarray}
&&H_0(u):=\frac12\int_{\mathbb R}u\rmd x,\label{momentum}\\
&&H_1(u):=\frac12\int_{\mathbb R}u^2\rmd x,\label{mass}\\
&&H_2(u):=-\frac{1}{2}\int_{\mathbb R}\left(uHu_x +\frac23u^3\right)\rmd x, \label{ener}\\
&&H_3(u):=\frac{2}{3}\int_{\mathbb R}\left(u_x^2+\frac32u^2Hu_x+\frac12u^4\right)\rmd x.\label{conser2}
\end{eqnarray}
The \eqref{eq: BO} may be viewed as a Hamiltonian system of the form
\begin{equation}\label{eq:BO Hamiltonian}
u_t=\mathcal{J}\frac{\delta H_2(u)}{\delta u},
\end{equation}
where $\mathcal{J}$ is the operator $\partial_x$, and $\frac{\delta H_2(u)}{\delta u}$ (or simply $H'_2(u)$) refers to the variational derivative of $H_2$ as follows
$$\bigg(\frac{\partial}{\partial\epsilon}H_{2}(u+\epsilon v)\bigg)\mid_{\epsilon=0}
=\int^\infty_{-\infty}\frac{\delta H_{2}}{\delta u}(x)v(x)dx.
$$

However, unlike the KdV equation \eqref{KdV}, the bi-Hamiltonian structure of \eqref{eq: BO} is quite tough \cite{FS88}. As the BO equation formulated in terms of two space operator $\partial_x$ and the Hilbert transform $H$, which makes the \eqref{eq: BO} share many features with completely integrable equations in two spatial dimensions. Let subscript $12$
denote the dependence on $x_1:=x$ and $x_2$, then for arbitrary functions $f_{12}$ and $g_{12}$, let us  define the following bilinear form:
\begin{equation}\label{eq:bilinear form}
\langle f_{12},g_{12}\rangle:=\int_{\R^2}f_{12}g^\ast_{12}\rmd x_1\rmd x_2,
\end{equation}
here the asterisk superscript denotes the complex conjugate in the rest of this manuscript.
Define the operators (in $L^2(\R^2,\mathbb C)$ with domain $H^1(\R^2,\mathbb C)$)
\begin{equation}\label{eq:functions}
\mathcal{u}^{\pm}_{12}:=u_1\pm u_2+i(\partial_{x_1}\mp\partial_{x_2}),\ u_j=u(x_j,t),\ j=1,2,
\end{equation}
then two compatible Hamiltonian operators associated with the BO equation are given by
\begin{equation}\label{eq:Hamilton operator}
\mathcal{J}^{(1)}_{12}:=\mathcal{u}^{-}_{12},\ \quad \mathcal{J}^{(2)}_{12}:=\big(i\mathcal{u}^{-}_{12}H_{12}-\mathcal{u}^{+}_{12}\big)\mathcal{u}^{-}_{12},
\end{equation}
where the operator $H_{12}$ is a generalized Hilbert transformation as follows
\begin{equation}\label{eq:extended Hilbert }
\big(H_{12}f_{12}\big)(x_1,x_2):=\frac{1}{\pi} \text{P.V.}\int_{-\infty}^{\infty}\frac{F(y,x_1-x_2)}{y-(x_1+x_2)}\rmd y,
\end{equation}
with $f_{12}(x_1,x_2)=F(x_1+x_2,x_1-x_2)$.
Then the BO hierarchy can be represented as follows \cite{FS88}:
\begin{eqnarray}
&&u_t=\frac{i}{2n}\int_{\R}\delta(x_1-x_2)\big(\mathcal{R}_{12}^\star\big)^n\mathcal{u}^{-}_{12}\cdot1\rmd x_2  \nonumber\\&&=\frac{i}{2n}\int_{\R}\delta(x_1-x_2)\mathcal{u}^{-}_{12}\mathcal{R}^n_{12}\cdot1\rmd x_2
=\mathcal{J}\frac{\delta H_n(u)}{\delta u}, \quad n\in \mathbb N. \label{eq:BO hierarchy}
\end{eqnarray}
where $\star$  denotes the adjoint with respect to the bilinear form \eqref{eq:bilinear form}. The recursion operator $\mathcal{R}_{12}$ and the adjoint recursion operator $\mathcal{R}_{12}^\star$ are defined by
\begin{equation}\label{eq:recursion operator}
\mathcal{R}_{12}:=\big(\mathcal{J}^{(1)}_{12}\big)^{-1}\mathcal{J}^{(2)}_{12},\ \ \mathcal{R}^\star_{12}:=\mathcal{J}^{(2)}_{12}\big(\mathcal{J}^{(1)}_{12}\big)^{-1}=i\mathcal{u}^{-}_{12}H_{12}-\mathcal{u}^{+}_{12},
\end{equation}
and in view of \eqref{eq:recursion operator}, they satisfy the following well-coupling condition
\begin{equation}\label{eq:well coupling}
\mathcal{R}^\star_{12}\mathcal{J}^{(1)}_{12}=\mathcal{J}^{(1)}_{12} \mathcal{R}_{12}.
\end{equation}
The first few equations of the BO hierarchy are then
\begin{eqnarray*}
&& u_t-u_x=0, \quad\text{for} \quad n=1,\quad \eqref{eq: BO}, \quad\text{for} \quad n=2;\\
&& u_{t}+\frac{4}{3}\left(u^3+\frac{3}{2}uHu_x+\frac{3}{2}H(uu_x)-u_{xx}\right)_x=0,\quad\text{for} \quad n=3.
\end{eqnarray*}

The energy space, where $H_2(u)$ is well-defined, is $H^{\frac12}(\R)$. The existence of global weak solutions $u\in C([0,+\infty);H^{\frac12}(\R))\cap C^1([0,+\infty);H^{-\frac32}(\R))$ was proved by Saut \cite{S79}. For strong solutions, Ionescu and Kenig \cite{IK07} showed the global well posedness for $s\geq0$ (see also the works of Tao \cite{T04} and Molinet and Pilod \cite{MP12} for global well posedness result in $H^1(\R)$ ). Such solution conserves $H_1$ and other conservation laws for suitable $s\geq0$. Concerning the weak continuity of the BO flow map, we refer to the work of \cite{CK10}.  Breakthrough has been made for the sharp low regularity well posedness theory of the (m)KdV and NLS equations \cite{KT18,KV19}, where the continuous family of the conservation laws below $L^2$ are established. For \eqref{eq: BO}, the conservation laws are achieved in $H^s(\R)$ by Talbut \cite{T21} for any $s>-\frac12$, the sharp low regularity global well posedness in $H^s(\R)$ with $s>-\frac12$ has been shown by G{\'e}rard, Kappeler and Topalov \cite{GKT20} on the torus and by Killip, Laurens and Visan \cite{KLV23} on the real line.

The BO equation \eqref{eq: BO} has soliton
 of the form
 \begin{equation}\label{eq:so}
u(t,x) =Q_{c}(x-ct-x_0),\ \ Q_{c}(s)=\frac{2c}{c^2s^2+1},\ c>0,\ x_0\in \R.
\end{equation}
By inserting \eqref{eq:so} into \eqref{eq: BO}, we have
\begin{equation}\label{eq:stationary}
-HQ_c'-Q_c^2+cQ_c=0,\quad c>0.
\end{equation}
Amick and Toland \cite{AT91}, Frank and Lenzmann \cite{FL13} showed that \eqref{eq:stationary} possesses a unique (up to symmetries) nontrivial $L^\infty$ solution. \eqref{eq: BO} exhibits even more complicated solutions called {\em{multi-solitons}}. The $m$-soliton solution is
characterized by the $2m$ parameters  $c_j$ and $x_{j}\ (j=1, 2, ..., m)$ as follows
 \begin{equation}\label{1.6a}
 U^{(m)}(t,x)=U^{(m)}(x-c_1t-x_{1}, x-c_2t-x_{2}, \ldots, x-c_mt-x_{m}).
 \end{equation}
Here ${\mathbf{c}}=(c_1,\ldots,c_m)$ is a collection of wave speeds satisfying the conditions $c_j>0, c_j\not=c_k$ for
$j\not=k \ (j, k=1, 2, ..., m)$ and ${\mathbf{x}}=(x_1,\ldots,x_m)$ is the initial position. The multi-soliton
 $U^{ (m)}$ has an explicit expression given by the tau function $f$ \cite{M06},
 \begin{equation}\label{1.7a}
 U^{ (m)}=U^{ (m)}(t,x;{\mathbf{c},\mathbf{x}})=i\frac{\partial}{\partial x}{\rm ln}\ \frac{f^*}{f}, \  f={\rm det}\ F,
 \end{equation}
where $F=(f_{jk})_{1\leq j,k\leq m}$ is an $m\times m$ matrix with elements
\begin{equation}\label{1.7b}
 f_{jk}=\left(x-c_jt-x_j+\frac{i}{ c_j}\right)\delta_{jk}
-\frac{2i}{c_j-c_k}(1-\delta_{jk}).
\end{equation}
Here, $f^*$ is the complex conjugate of $f$ and $\delta_{jk}$ is the Kronecker's function. The expression \eqref{1.7a} shows that the BO multi-solitons exhibit no phase shift after the soliton collisions. Moreover, for large time $t$, the BO $m$-solitons can be represented by a superposition of $m$ algebraic solitons as follows
\begin{equation}\label{eq:asympotic behavior}
\lim_{t\rightarrow +\infty}\left\|U^{ (m)}((t,\cdot;{\mathbf{c},\mathbf{x}})-\sum_{n=1}^mQ_{c_j}(\cdot-c_jt-x_j)\right\|_{H^s(\R)}=0,\quad\  s\in \mathbb{N}.
\end{equation}

Over the past four decades, there are many known results  associated with the stability characteristics of the BO
solitons and multi-solitons. A spectral stability analysis of the solitons has been given by Chen and Kaup \cite{CK80};  The spectral stability of the general $m$-solitons was shown in \cite{MK97}; The orbital (i.e. up to translations) stability of one soliton in the energy space $H^{\frac12}(\R)$ was established in \cite{BBSDB83,W87}. Moreover, stability of solitons for two classes of nonlinear dispersive equations (consist of \eqref{eq: ILW} and BBM equations with general power type nonlinearity) were also investigated in \cite{W87}, see also \cite{B72} for earlier stability results.
Orbital stability of  double solitons in $H^1(\R)$ as critical points of the constrained Hamiltonian $H_{3}(u)$ was showed in \cite{LN}. The stability in $H^{\frac12}(\R)$ of sum of widely separated solitons was considered in \cite{GTT09,KM09} and the asymptotic stability of sum of $m$ solitons is established by Kenig and Martel \cite{KM09} by employing the approach of \cite{MMT02}. For the generalized Benjamin-Ono equation, there are interesting results concerning the asymptotic stability and blow up of their solutions \cite{KMR11,MP17}. The existence and uniqueness (for mass supercritical BO) of strongly interacting multi-solitons (multi-pole type solutions) for a generalized BO equation has been shown recently by the authors \cite{LW23}. For \eqref{eq: BO}, there is no multi-pole solutions since its eigenvalue problem possesses only finite and simple eigenvalues \cite{W16}.  We refer to \cite{S19} for a very nice exposition for the above related issues.

In this manuscript we aim to show the following dynamical stability of arbitrary $m$-solitons of the BO equation. As the BO equation is more likely a $2$d integrable system, our approach opens the way to treat the stability problems of multi-solitons for other completely integrable models like \eqref{eq: ILW}(even for some 2d integrable models like KP-I equation). Moreover, our approach can also give alternative proofs for the stability of multi-solitons of the KdV and mKdV equations \cite{MS93,LW}. The main result of this manuscript is as follows.

\begin{theorem}\label{thm1.1}
Given $m\in\mathbb N,$ $m\geq1$, a collection of wave speeds ${\mathbf c}=(c_1,\cdot\cdot\cdot,c_m)$ with $0<c_1<\cdot\cdot\cdot<c_m$ and a collection of space transitions $\mathbf x=(x_1,\cdot\cdot\cdot,x_m)\in\R^m$, let $U^{ (m)}(\cdot,\cdot;{\mathbf c},{\mathbf x})$ be the corresponding multi-solitons of \eqref{eq: BO}. Then for any $\epsilon>0$, there exists $\delta>0$ such that  for any $u_0\in H^{\frac{m}{2}}(\R)$, the following stability property holds. If
 \[
\|u_0-U^{ (m)}(0,\cdot;{\mathbf c},{\mathbf x})\|_{H^{\frac{m}{2}}}<\delta,
\]
then for any $t\in\R$ the corresponding solution $u$ of~\eqref{eq: BO} verifies
 \[
\inf_{\tau\in\R,\ {\mathbf y}\in\R^m}\|u(t)-U^{ (m)}(\tau,\cdot;{\mathbf c},{\mathbf y})\|_{H^{\frac{m}{2}}}<\epsilon.
\]
\end{theorem}

As a direct consequence, we give a new proof of the orbital stability of the double solitons in \cite{LN}. The main differences lie in the spectral analysis part in Section \ref{sec_4} (see Corollary \ref{co3.10} and Remark \ref{re:4.3} for details).

\begin{corollary}\label{Co1.1}\cite{LN}
The \eqref{eq: BO} double solitons $U^{(2)}(t,x)$ is orbitally stable in $H^1({\R})$.
\end{corollary}

\begin{remark} \label{re1.4} There are some interesting results of the stability and asymptotic stability of trains of $m$ solitons  for the BO equations obtained in \cite{GTT09,KM09}. Such type of stability (which holds also for other non-integrable
models, see \cite{MMT02} for subcritical gKdV equations) usually does not include the dynamical stability of $m$-solitons as in Theorem \ref{thm1.1}. We get the stability of the whole orbit of $m$-solitons for all the  time by minimizing the conserved quantities.
\end{remark}

We employ the approach from the stability analysis of the
multi-solitons of the KdV equation by means of
 variational argument \cite{MS93}. It is demonstrated that the Lyapunov functional $S_m$ of the BO $m$-solitons profile $U^{ (m)}(x)=U^{ (m)}(0,x)$ is given by (see also \cite{M06})
\begin{equation}\label{Lyafun}
S_m(u)=H_{m+1}(u)+\sum_{n=1}^m\mu_nH_{n}(u),
\end{equation}
and $\mu_n$ are  Lagrange multipliers which will be expressed in terms of the elementary
symmetric functions of $c_1,c_2, ..., c_m$. We refer to Section \ref{sepecsec2}
for more details.
Then we show that $U^{ (m)}$ is a critical point of the
functional $S_m$.
Using \eqref{Lyafun}, this condition can be written as the following Euler-Lagrange equation
\begin{equation}\label{E-L}
\frac{\delta H_{m+1}(u)}{\delta u}+\sum_{n=1}^m\mu_n\frac{\delta H_{n}(u)}{\delta u}=0, \ {\rm at}\ u=U^{ (m)}.
\end{equation}
The dynamical stability of $U^{ (m)}$ is implied by the fact that $U^{ (m)}(x)$ is a minimizer of
the functional $H_{m+1}$ under the following $m$ constraints
\begin{equation}\label{eq:constrain2}
H_{n}(u)=H_{n}(U^{ (m)}), \quad n=1, 2, ..., m,
\end{equation}
which requires that the self-adjoint second variation operator of $S_m$
\begin{eqnarray}\label{eq:linearized n-soliton operator}
\mathcal L_m:=S_m''(U^{ (m)}),
\end{eqnarray}
is strictly positive if one modulates the directions given by the constraints. We mention here that $\mathcal L_m$ is highly nonlocal since the Hilbert transform $H$ is involved.

As a byproduct of showing Theorem \ref{thm1.1}, one can express the negative eigenvalues of the isoinertial operator $\mathcal L_m$ \eqref{eq:linearized n-soliton operator} explicitly in terms of the wave speeds $\{c_j\}_{j=1}^m$. Similar result for the KdV equation was shown in \cite{W22}.
\begin{theorem}\label{thm1.3}
The linearized operator $\mathcal {L}_m$  around the $m$-solitons possesses $[\frac{m+1}{2}]$ negative eigenvalues $\nu_k$, $k=1,2,\cdot\cdot\cdot,[\frac{m+1}{2}]$, where $[x]$ is the largest integer not exceeding $x$. Moreover, for each $k$ and $j=1,2,\cdot\cdot\cdot,m$, there exist constants $C_k>0$, independent of the wave speeds $c_1,\cdot\cdot\cdot,c_m$, such that
\begin{eqnarray}\label{eigenvalue1.3}
\nu_k=-C_kc_{2k-1}\prod_{j\neq 2k-1}^m(c_j-c_{2k-1}),\quad k=1,2,\cdot\cdot\cdot,[\frac{m+1}{2}].
\end{eqnarray}
\end{theorem}

The ideas developed by Maddocks and Sachs  have been successfully implemented to obtain stability results in various settings. Neves and Lopes~\cite{LN} proved the stability of double solitons of the BO equation, but it seems that their approach did not handle the arbitrary $m$-soliton. Le Coz and the second author ~\cite{LW} proved the stability of $m$-solitons of the mKdV equation, meanwhile, a quasi-linear integrable model called Camassa-Holm equation was considered by the second author and Liu \cite{LW20}, where stability of smooth multi-solitons is proved by employing some inverse scattering techniques.
We also mention the work of Kapitula~\cite{K07}, which is devoted to the stability of $m$-solitons of a large class of integrable systems, including in particular the cubic nonlinear Schr\"odinger equation. Very recently, a variational approach was used by Killip and Visan~\cite{KV20} to obtain the stability of KdV multi-solitons in  $H^{-1}(\R)$.  Stability results in low regularity  $H^s$ with $s>-\frac12$ were also obtained by Koch and Tataru~\cite{KoTa20} for multi-solitons of both the mKdV equation and the cubic nonlinear Schr\"odinger equation, the proof of which relies on an extensive analysis of an iterated B\"acklund transform. It is remarkable that ~\cite{KoTa20} also proved the stability of the multi-pole solutions of mKdV and cubic nonlinear Schr\"odinger equations. The major difference between the approach \cite{LW} and the approaches of \cite{MS93}, \cite{LN} lie in the analysis of spectral properties. Indeed, the spectral analysis of Maddocks and Sachs and many of their continuators relies on an extension of Sturm-Liouville theory to higher order differential operators (see~\cite[Section 2.2]{MS93}). As the BO equation is nonlocal, Neves and Lopes~\cite{LN} were lead to introduce a new strategy relying on isoinertial properties of the linearized operators around the $m$-solitons $\mathcal {L}_m$ for $m=2$. That is to say, the spectral information of $\mathcal {L}_2$ is independent of time $t$. Therefore, one can choose a convenient $t$ to calculate the inertia and the best thing we can do is to calculate the inertia $in(\mathcal {L}_2(t))$  as $t$ goes to $\infty$. However, in \cite{LN}, the approach of their spectral analysis for higher order linearized operators around one solitons can not be applied for large $m$.

To handle this issue, in \cite{LW}, we adapt the ideas of \cite{MS93} and \cite{LN} and develop a method to treat the  spectral analysis of linearized operators around arbitrary $m$-solitons. The main ingredient is to show some conjugate operator identities to prove the spectral information of the linearized operator around the multi-solitons. Such  conjugate operator identities are established by employing the recursion operator of the equations. In particular, let $\varphi_{c}$  be the one soliton profile with wave speed $c>0$ of the KdV or mKdV equation. The conservation laws of the equations denoted by $H_{K,n}$ (the subscript $K$ denotes the (m)KdV) for $n\geq1$. Then the linearized operator around the one soliton $H_{K,n+1}''(\varphi_{c})+cH_{K,n}''(\varphi_{c})$  can be diagonalized to their constant coefficient counterparts by employing the following auxiliary operators $M$ and $M^t$:
\begin{eqnarray*}
  M:=\varphi_{c}\partial_x\left(\frac{\cdot}{\varphi_{c}}\right),\quad M^t=-\frac{1}{\varphi_{c}}\partial_x\left(\varphi_{c}\,\cdot\,\right),
\end{eqnarray*}
the following conjugate operator identity holds:
 \begin{eqnarray}\label{Mt03}
M\bigg(H''_{K,n+1}(\varphi_{c})+cH''_{K,n}(\varphi_{c})\bigg)M^t=M^t\bigg((-\partial_x^2)^{n-1}(-\partial_x^2+c)\bigg)M.
\end{eqnarray}
 The recursion operator plays an important role in showing \eqref{Mt03} as it can not be computed by hand when $n$ is large. Such method is valid for a large amount of 1d completely integrable models which possess explicit recursion operators. However,  the BO equation is more similar to a $2$d completely integrable model and has no explicit recursion operators \eqref{eq:recursion operator}. Indeed,  as stated in Zakharov
and Konopelchenko \cite{ZK84}, recursion operators seem to exist explicitly only in $1$d integrable systems. Hence, the approach in \cite{LW} can not be directly applied for the BO equation.

To extend the spectral theory of Neves and Lopes~\cite{LN} to an arbitrary number $m$ of composing solitons, which leads to increasing technical complexity (inherent to the fact that the number of composing solitons is now arbitrary), no major difficulty arises here since which has been done in ~\cite{LW}. Then our main task was to implement this spectral theory for the multi-solitons of~\eqref{eq: BO}. At that level, we had to overcome major obstacles coming from the non-locality of the linearized operators. The conjugate type operator identities \eqref{Mt03} are usually wrong or very difficult to check.  To deal with the arbitrary $m$ case, it is necessary to acquire a deeper understanding of the relationships between $m$-solitons, the variational principle that they satisfy, and the spectral properties of the operators obtained by linearization of the conserved quantities around them. In particular, we need to have a good knowledge of the spectral information of the higher order linearized Hamiltonian $L_n:=H_{n+1}''(Q_c)+cH_n''(Q_c)$ for all $n\geq1$. To show the spectral information of such higher order linearized operators, to the best knowledge, there is no good way except the conjugate operator identity approach in the literature. In addition, as we stated before, it is impossible to prove the conjugate type operator identities \eqref{Mt03} for  large $n$, since the \eqref{eq: BO} possesses no explicit recursion operator ( the conjugate type operator identity is quite involved even for $n=2$ which achieved by brute force in ~\cite{LN}).

To overcome this difficulty, we present an approach for the spectral analysis of the linearized operators $L_n$ is as follows: Firstly, we derive the spectral information of the operator $\mathcal{J}L_n$, which is easier than to have the spectral information of $L_n$, the reason is that the operator $\mathcal{J}L_n$ is commutable with the adjoint recursion operator. The spectral analysis of the adjoint recursion operator is possible since we can solve the eigenvalue problem of the BO equation; Secondly, we show that the eigenfunctions of $\mathcal{J}L_n$ plus a generalized kernel of $\mathcal{J}L_n$ form an orthogonal basis in $L^2(\R)$, which can be viewed as a completeness or closure relation. Lastly, we calculate the quadratic form $\langle L_nz,z\rangle$ with function $z$ that has a decomposition in the above basis, then the spectral information of $L_n$ can be derived directly. We believe this approach can even be applied to some $2$d integrable models like KP-I equation.

The reminder of the paper is organized as follows. In Section \ref{sepecsec2}, we summarize some basic properties
of the Hamiltonian formulation of the BO equation and present some results with the help of IST,
which provide some necessary machinery in carrying out the spectral analysis.  Section \ref{sec_4} is devoted to a detailed spectral analysis of $\mathcal L_m$, the Hessian operator of $S_m$. The proof of Theorem \ref{thm1.1}, the dynamical stability of the $m$-soliton
solutions of the BO equation,  and Theorem ~\ref{thm1.3} will be given in Section \ref{sec_5}.

\section{Background results for the BO equation }\label{sepecsec2}
\setcounter{equation}{0}
In this section we collect some preliminaries in showing Theorem \ref{thm1.1}. This Section is divided into four parts. At the first part, we review some basic properties of the Hilbert transform $H$ and the generalized Hilbert transform $H_{12}$ defined in \eqref{eq:extended Hilbert }. Secondly, we present the equivalent eigenvalue problem of the BO equation and the basic facts of which through the inverse scattering transform. The conservation laws and trace formulas of the BO equation are also derived. In Subsection \ref{sec_2.3}, we recall the Euler-Lagrange equation of the BO multi-solitons in \cite{M06}, which admits a variational characterization of the $m$-soliton profile $U^{ (m)}(x)$. Subsection \ref{sec_2.4} is devoted to the investigation of the bi-Hamiltonian formation of the BO equation, the recursion operators are introduced to the computation of the conservation laws at the multi-solitons. Moreover, an iteration formula of the linearized operators $H_{n+1}''(Q_c)+cH_n''(Q_c)$ for all $n\in\mathbb N$ is established, it follows that investigating the properties of recursion operators (even if they are not explicit)  contributes the major difficulty of the spectral analysis issue.
\subsection{Some properties of the Hilbert transform}\label{sec_2.0}
For the reader's convenience, we review here some elementary properties of the Hilbert transform $H$ and the generalized  Hilbert transform $H_{12}$ (defined in \eqref{eq:extended Hilbert })  that figured in the forthcoming analysis.
It is demonstrated that for $f\in L^2(\R)$ implies $Hf\in L^2(\R)$ and the Fourier transform of $Hf$
\[ \widehat{Hf}(\xi)=i sgn(\xi)\hat{f}(\xi), \ \text{where}\ sgn(\xi)\xi=|\xi|,  \ \text{for all }\ \xi\in \R.\]
It is clear that $H^2f=-f$ for $f\in L^2(\R)$  and $H\partial_xf=\partial_xHf$ for $f\in H^1(\R)$. Moreover, the operator $H$ is skew-sdjoint in the sense that
\[\langle Hf,g\rangle=-\langle f,Hg\rangle,\]
and maps even functions into odd functions and conversely.

A useful property bears upon the Hilbert transformation of a function $f^+$ ($f^-$) analytic in the upper (lower) half complex plane and vanishing at $\infty$, in this case, one has
\begin{equation}\label{Hul}
Hf^{\pm}=\pm i f^{\pm}.
 \end{equation}

There is a parallel theory upon the generalized  Hilbert transform $H_{12}$ \eqref{eq:extended Hilbert }, for more details we refer to \cite{FS88}. Let $f_{12}=f(x_1,x_2)\in L^2(\R^2,\mathbb C)$  be the function depend on $x_1=x$ and $x_2$, then we see that
\[
H^2_{12}=-1, \ H^{\ast}_{12}=-H_{12}, \text{and}\ \partial_{x_j}H_{12}f_{12}=H_{12}\partial_{x_j}f_{12},\ j=1,2.
\]
Moreover, for any $g\in L^2(\R)$, there holds
\[
H_{12}g(x_j)=H_jg(x_j),\  H_jf(x_i,x_j):=\frac{1}{\pi} \text{P.V.}\int_{-\infty}^{\infty}\frac{f(x_i,y)}{ y-x_j}\rmd y,\ i\neq j.
\]
If $f_{12}^{(\pm)}:=\pm\frac{1}{2}(1\mp iH_{12})f_{12}$, then $f_{12}^{(+)}$ and $f_{12}^{(-)}$ are holomorphic for $\imp(x_1 +x_2)>0$
and $\imp(x_1+x_2)<0$, respectively. Moreover, one has
 \begin{equation}\label{Hu2}
H_{12}\big(f_{12}^{(+)}-f_{12}^{(-)}\big)=i\big(f_{12}^{(+)}+f_{12}^{(-)}\big).
 \end{equation}
\subsection{Eigenvalue problem and conservation laws}\label{sec_2.1}
The Benjamin-Ono equation can be solved by inverse scattering transform. Here, we list some results
related to the theory of the inverse scattering transform for the Benjamin-Ono equation, which are necessary for our stability analysis. We refer to \cite{CW90,FA83,KLM99,KM98,M06,W16,W17} for detailed proof of such results.

We fix a real valued function $u=u(t,x)$ on $\mathbb{R}\times \mathbb{R}$, such that for $t$ $u(t,x)$ has a good enough decay for $|x|\rightarrow+\infty$ . We also define the projection operators $P_\pm$ as follows: $P_{\pm}:=\pm\frac{1}{2}(1\mp iH)$ (therefore $P_+-P_-=1$). Let $\lambda$ be the eigenvalue (or the spectral parameter) and $\gamma$ be a constant to be chosen later. Now, we can consider the following eigenvalue problem
\begin{align}
&i\phi_x^++\lambda(\phi^+-\phi^-)=-u\phi^+,\ \lambda\in \R;\label{space lax}\\
&i\phi_t^{\pm}-2i\lambda\phi_x^{\pm}+\phi_{xx}^{\pm}-2iP_{\pm}(u_x)\phi^{\pm}=-\gamma\phi^{\pm}.\label{time lax}
 \end{align}
where for all fixed $t$, $\phi^+(t)$ (or $\phi^-(t)$, respectively) is the boundary value of some analytic function on the upper half complex plane $\mathbb{C}^+$ (or on the lower half complex plane $\mathbb{C}^-$, respectively). We define the Jost
solutions $N,{\bar N}, M,{\bar M}$ associated to \eqref{space lax} be functions in $(x,\lambda)$ satisfying
\begin{equation}\label{2.4b}
\begin{split}
&N_x-i\lambda N=iP_+(uN), \\
&\bar N_x-i\lambda \bar N=iP_+(u\bar N)-i\lambda,\\
&M_x-i\lambda M=iP_+(u M)-i\lambda, \\
&\bar M_x-i\lambda \bar M=iP_+(u\bar M),
\end{split}
\end{equation}
and the following boundary conditions
\begin{gather}
\label{Jost be1} \lim_{x\rightarrow+\infty}\left(|N(x,\lambda)- e^{i\lambda x}|+|\bar N(x,\lambda)-1|\right)=0,\\
\label{Jost be2}  \lim_{x\rightarrow-\infty}\left( |M(x,\lambda)- 1|+ |\bar M(x,\lambda) - e^{i\lambda x}|\right)=0.
\end{gather}

It is not hard to see that the Jost solutions satisfy
\begin{equation}\label{two jost re}
M=\bar N+\beta N,
\end{equation}
where $\beta$ is the reflection coefficient given by
\begin{equation*}
\beta(\lambda)=i\int_{\R}u(y)M(y,\lambda)e^{-i\lambda y}\rmd y.
\end{equation*}
It is inferred from \cite{KLM99} that the asymptotic behaviors of $N,\bar{N}$ and $M$ are given by :
 \begin{eqnarray}
&&\left|N(x,\lambda)-\frac{1}{\Gamma(\lambda)}e^{i\lambda x}\right|\rightarrow0,\ x\rightarrow -\infty,\ \Gamma(\lambda):=e^{\frac{1}{2\pi i}\int_0^{\lambda}\frac{|\beta(k)|^2}{k}\rmd k}; \label{2.N}\\
&&\left|\bar N(x,\lambda)-\left(1-\frac{\beta(\lambda)}{\Gamma(\lambda)}e^{i\lambda x}\right)\right|\rightarrow 0,\ x\rightarrow -\infty;\label{2.barN}\\
&&\left|M(x,\lambda)-( 1+\beta(\lambda)e^{i\lambda x})\right|\rightarrow0,\ x\rightarrow +\infty.\label{2.M}
\end{eqnarray}
There exist discrete eigenfunctions $\Phi_j(x)\in P_{+}(H^1(\R))$ associated to negative eigenvalues
$\lambda_j$ for $j=1, 2, ..., m$ (we mention here $m$ must be finite and $\lambda_j$ is simple, due to \cite{W16}), which satisfy the equation
\begin{equation}\label{equan phi}
\partial_{x}\Phi_{j}-i\lambda_j\Phi_j=iP_+(u\Phi_j), \ j=1, 2, ..., m,
\end{equation}
and the boundary conditions
\begin{equation}\label{asy phi}
\Phi_j(x) \sim \frac{1}{x}, \ x \rightarrow +\infty,
\ j=1, 2, ..., m.
\end{equation}
By using the Fredholm theory, Fokas and Ablowitz \cite{FA83} show that when $\lambda\rightarrow\lambda_j$, for some $j\in\{1,2,\ldots,m\}$, we have
$$\bar N(x,\lambda)\sim M(x,\lambda)= - \frac{i\Phi_j(x)}{\lambda-\lambda_j}+(x+\gamma_j)\Phi_j(x)+O\left(|\lambda-\lambda_j|\right).$$
Here the complex-valued constants $\gamma_j$ are called {\em{normalization constants}}. Moreover, we have
\begin{equation}\label{im normal}
\imp \gamma_j=-\frac{1}{2\lambda_j}=\frac{1}{c_j}.
\end{equation}
The set
\begin{equation} \label{eq sacttering data}
\mathcal{S}:=\left\{(\beta(\lambda),\lambda_1,\ldots,\lambda_m):\,\lambda>0\right\}
\end{equation}
 is called the \emph{scattering data}.
In particular, when $u$ is a soliton potential given by \eqref{eq:so}, one has that $\beta(\lambda)\equiv0$  and the corresponding Jost solutions can be computed explicitly. In this case, one has
\[
\lambda_1=-\frac{c}{2},\ \gamma_1=-x_0+\frac{i}{c}.
\]
Then it reveals from \eqref{equan phi} and \eqref{asy phi} that
\begin{eqnarray}
&&\Phi_1(x)=\frac{1}{x+\gamma_1};\label{phi1}\\
 &&\bar{N}(x,\lambda)=M(x,\lambda)=1-\frac{i\Phi_1(x)}{\lambda-\lambda_1};\label{phi2}\\
&&  N(x,\lambda)=e^{i\lambda x}\left(1+\frac{i}{\lambda_1}\Phi_1(x)\right).\label{phi3}
\end{eqnarray}

Let us compute the conservation laws of the BO equation. It follows from \eqref{2.4b} and  \eqref{time lax} (by choosing $\gamma=0$) that,
\begin{equation}\label{iteration dens}
\bar N_t-2\lambda \bar N_{x}-i\bar N_{xx}-2(P_+u_x)\bar N =0,
\end{equation}
therefore, the integral $\int^\infty_{-\infty}u(x,t)\bar N(x,t)dx$
is independent of time.  Expanding $\bar N$ as a powers series of $\lambda^{-1}$
\begin{equation*}
\bar N=\sum_{n=0}^\infty\frac{(-1)^n\bar N_{n+1}}{ \lambda^n}, \ \bar N_1=1,
\end{equation*}
and inserting it into \eqref{2.4b}, we obtain the following recursion relations of
$\bar N_n$:
\begin{equation}\label{iteration conser}
\bar N_{n+1}=i\bar N_{n,x}+P_+(u\bar N_n), \ n\geq 1.
\end{equation}
Therefore, the higher order conservation laws can be calculated as follows
\begin{equation*}
I_n(u)=(-1)^n\int^\infty_{-\infty}u\bar N_ndx.
\end{equation*}
The {\em{trace identities}} describes the relation between the conservation laws $I_n$ and the scattering data $(\beta(\lambda),\lambda_1,\ldots,\lambda_m)$:
\begin{equation}\label{trace for I}
I_n(u)=(-1)^n\left\{2\pi\sum_{j=1}^m(-\lambda_j)^{n-1}+\frac{(-1)^n}{2\pi}
\int^\infty_0\lambda^{n-2}|\beta(\lambda)|^2d\lambda\right\},
\ n=1, 2, ...,
\end{equation}
for $u\in L^2(\R,(1+x^2)dx)\cap L^{\infty}(\R)$.
 The first term on the right-hand side of \eqref{trace for I} is the contribution
of solitons while the second term comes from radiations.
In terms of $I_n$, the conservation laws $H_n$ presented in Section \ref{intro} can be expressed as follows:
 \begin{equation}\label{HI relation}
H_n=\frac{2^{n-1}}{n}I_{n+1}, \ \text{for all} \ n\geq 1.
\end{equation}
The first four of $H_n$ except $H_0$ are explicitly given by \eqref{mass}, \eqref{ener} and \eqref{conser2}.
It is inferred from \eqref{trace for I} that
\begin{equation}\label{trace for H}
H_n=\frac{(-1)^{n+1}}{n}\left\{\pi\sum_{j=1}^m(-2\lambda_j)^{n}+\frac{(-1)^{n+1}}{ 2\pi}
\int^\infty_0(2\lambda)^{n-1}|\beta(\lambda)|^2d\lambda\right\},
\ n=1, 2, ....
\end{equation}

Similar to the KdV equation case, the BO conservation laws  are in involution,
i.e., $H_n$ $(n=0,1, 2, ...)$ commute with each other  in the following Poisson bracket
\begin{equation*}
\int^\infty_{-\infty}\left(\frac{\delta H_n}{ \delta u}(x)\right)\bigg|_{u=U^{ (m)}}
\frac{\partial}{\partial x}\left(\frac{\delta H_l}{\delta u}(x)\right)\bigg|_{u=U^{ (m)}}dx=0,
\ n, l=0,1, 2, ...\ .
\end{equation*}
Note that $H_0$ is the unique Casimir function of \eqref{eq: BO}.
\par
\subsection{ The Euler-Lagrange equation of the $m$-solitons profile}\label{sec_2.3}
In order to show the dynamical stability of the BO $m$-solitons,
we need the formulas of  the variational derivatives of $H_n$ at
the $m$-soliton potential $U^{ (m)}(t,x)$. Using the explicit expression \eqref{1.7a} for the BO $m$-solitons, it would in theory be possible to verify by hand for any given $m$ that they also satisfy variational principles. However, the calculations would
 rapidly become unmanageable when $m$ grows. In \cite{M06}, Matsuno provided an algebraic proof for this fact. For sake of completeness, we give an overview of the results and proof in \cite{M06} %
\footnote{We mention here our conservation laws are sightly modified (see \eqref{HI relation}) with respect to the conservation laws in \cite{M06,LN}.}
.

The variational derivative of the discrete eigenvalues with respect to the
potential (at $m$-solitons profile) is given by
\begin{equation}\label{2.12}
\left(\frac{\delta \lambda_j}{ \delta u}(x)\right)\bigg|_{u=U^{ (m)}}=\frac{1}{2\pi\lambda_j}\Phi_j^*(x)\Phi_j(x),
\ j=1, 2, ..., m.
\end{equation}
Here, the eigenfunction $\Phi_j$  corresponding to the eigenvalue $\lambda_j$ satisfies the following equation
\begin{equation}\label{2.13}(x+\gamma_j)\Phi_j+i\sum_{k\neq j}^m\frac{1}{ \lambda_j-\lambda_k}\Phi_k=1,
\ j=1, 2, ..., m,
\end{equation}
where $\gamma_j=-x_{j}-\frac{i}{2\lambda_j}$ and $x_{j}$ are real constants and $\lambda_j=-\frac{c_j}{2}, j=1, 2, ..., m$.
Recall that the reflection coefficient $\beta(\lambda)=0$ when $u=U^{ (m)}$,
we use \eqref{trace for H} and \eqref{2.12} to obtain the variational derivatives of $H_n$  at $u=U^{(m)}$:
\begin{equation}\label{2.14}\left(\frac{\delta H_n}{ \delta u}(x)\right)\bigg|_{u=U^{ (m)}}=(-1)^{n+1}
2\sum_{j=1}^m(-2\lambda_j)^{n-2}\Phi_j^*(x)\Phi_j(x),
\ n= 1,2, 3, ..., m.
\end{equation}
The $m$-solitons profile $U^{ (m)}(0,x)$ has the following two alternative expressions \cite{M06}:
\begin{eqnarray}
&&U^{ (m)}=i\sum_{j=1}^m(\Phi_j-\Phi_j^*),\ \
U^{ (m)}=-\sum_{j=1}^m\frac{1}{\lambda_j}\Phi_j^*\Phi_j,\label{2.16}
\end{eqnarray}
which immediately implies that $U^{ (m)}(x)>0$  since discrete eigenvalues $\lambda_j=-\frac{c_j}{2}<0$.

On the other hand, the variational derivative of $\beta$ with respect to $u$ is given by
\begin{equation*}\label{2.17}
\frac{\delta\beta(\lambda)}{\delta u}(x)=iM(x,\lambda)N^*(x,\lambda).
\end{equation*}
When $u=U^{ (m)}$, one has  $\beta\equiv 0$ and therefore $M\equiv\bar N $ by \eqref{two jost re}.
We also have the the following orthogonality conditions for the  function $MN^*$
\begin{equation}\label{2.18}\int^\infty_{-\infty}M(x,\lambda)N^*(x,\lambda)\frac{\partial}{\partial x}
\left(\Phi_j^*(x)\Phi_j(x)\right)dx=0, j=1, 2, ..., m.
\end{equation}
Similarly, the  variational derivative of the normalization constants $\gamma_j$ ($j=1,2,\cdot\cdot\cdot,m$) with respect to $u$
is given by
\begin{eqnarray}
\frac{\delta \gamma_j}{\delta u}(x)&=&-\frac{1}{2\pi \lambda_j^2}(x+\gamma_j)\Phi_j^*\Phi_j+i\sum_{l\neq j}\frac{\Phi_j^*\Phi_l-\Phi_l^*\Phi_j}{2\pi \lambda_j(\lambda_l-\lambda_j)^2}\nonumber\\&+&\frac{1}{4\pi^2i \lambda_j}\int_0^{+\infty}\frac{\big(\beta(\lambda)\Phi_j^*N-\beta^\ast(\lambda)\Phi_jN^*\big)\rmd \lambda}{(\lambda-\lambda_j)^2}.\label{deriva gamma}
\end{eqnarray}
The results presented above are derived
by the IST of the BO equation,  especially through the analysis of the eigenvalue problem \eqref{space lax} of the Lax pair, we refer to \cite{FA83,KLM99,M06} for more details.

Using the above formula,  we can obtain the variational characterization of the BO $m$-solitons profile proved by Matsuno \cite{M06}. Here we provide an alternate proof for the last step in this approach:
\begin{proposition} \label{pr2.2}\cite{M06}  The profiles of the BO $m$-solitons  $U^{ (m)}$ satisfy \eqref{E-L} if the Lagrange multipliers $\mu_n$ are symmetric functions of the wave speeds $c_1,c_2,\cdot\cdot\cdot,c_m$ which satisfy the following:
\[
\prod_{n=1}^m(x+c_n)=x^m+\sum_{n=1}^m\mu_nx^{m-n}, \quad x\in \R.
\]
In particular,
$\mu_n$ are given by the following Vieta's formulas: for $k=1,\dots,m$
 \begin{equation}
\mu_{m+1-k}=\sum_{1\leq i_{1}<\cdots<i_{k}\leq
  m}\left(\prod_{j=1}^{k}c_{i_j}\right).\label{eq:vieta}
\end{equation}
\end{proposition}

\begin{proof}  Let $\Psi_j=\Phi_j^*\Phi_j$ be squared eigenfunctions and $c_j=-2\lambda_j$ be the wave speeds.
We deduce from \eqref{E-L} and \eqref{2.14} to have the following linear relation among $\Psi_j$
\begin{equation*}\label{3.1}
\sum_{j=1}^mc_j^{m-1}\Psi_j + \sum_{n=1}^m(-1)^{m-n+1}\mu_n
\sum_{j=1}^mc_j^{n-2}\Psi_j=0.
\end{equation*}
Due to the fact that $\Psi_j$ are linearly independent,
$\mu_n$ must satisfy the following system of linear algebraic equations:
\begin{equation*}\label{3.2}
\sum_{n=1}^m(-1)^{m-n}c_j^{n-1}\mu_n=c_j^{m}, j=1, 2, ..., m.
\end{equation*}
As a consequence, we see that for each $j=1,\dots, m$, we have
\[
(-c_j)^{m}+\sum_{n=1}^m\mu_{n}(-c_j)^{n-1}=0,
\]
which implies that $-c_j$ are the roots of the polynomial $x^m+\sum_{n=1}^m\mu_n x^{n-1}=0$.  Since $c_1<\cdots<c_m$, we obtain ~\eqref{eq:vieta} from Vieta's formula immediately.
\end{proof}

\subsection{Bi-Hamiltonian formation of \eqref{eq: BO}}\label{sec_2.4}

In viewing of \eqref{eq:BO hierarchy}, we can define the recursion operator from the following relations for the variational derivatives of conservation laws $H_n(u):H^{\frac{n-1}{2}}(\R)\rightarrow \R $ $(n\in \mathbb N)$ with respect to $u$,
 \begin{equation} \label{eq:recursion2}
\frac{\delta H_{n+1}(u)}{\delta u}=\mathcal{R}(u)\frac{\delta H_{n}(u)}{\delta u},
\end{equation}
unlike the KdV case, the recursion operator $\mathcal{R}(u)$ is implicit and should be understood from \eqref{eq:recursion operator}.
The adjoint operator of  $\mathcal{R}(u)$ is
\begin{equation}\label{eq:adjoint of R}
\mathcal{R}^{\star}(u)=\mathcal{J}\mathcal{R}(u)\mathcal{J}^{-1},
\end{equation}
and it is not difficult to see that the operators $\mathcal{R}(u)$ and $\mathcal{R}^{\ast}(u)$ satisfy
\begin{equation}\label{eq:R Rstar}
\mathcal{R}^{\star}(u)\mathcal{J}=\mathcal{J}\mathcal{R}(u).
\end{equation}
The above definitions of recursion operators are reasonable since $\mathcal{R}(u)$ maps the variational derivative of conservation laws of \eqref{eq: BO} onto the variational derivative of conservation laws, $\mathcal{R}^{\star}(u)$ maps infinitesimal generators of symmetries of \eqref{eq: BO} onto infinitesimal generators of symmetries.
The starting symmetry of \eqref{eq: BO} is $u_x$ \cite{FF81}, therefore, \eqref{eq:adjoint of R} is well-defined since
\[\big(\mathcal{R}^{\star}(u)\big)^n u_x=\mathcal{J}\big(\mathcal{R}(u)\big)^nH_1'(u)=\mathcal{J}\big(\mathcal{R}(u)\big)^nu, \quad n\in \mathbb N.\]
For future reference, we need to show the above definition of $\mathcal{R}(u)$ is unique and differentiable with respect to $u$. For KdV equation \eqref{KdV}, its recursion operator is explicit, the uniqueness and smoothness of which can be checked directly. In particular,  we consider \eqref{KdV} with $\delta=3$ and for functions defined on Schwartz space $\mathcal S(\R)$ for simplicity, the recursion operator of \eqref{KdV} is
$\mathcal{R}_K(u):=-\partial_x^2-\frac{2}{3}u-\frac{2}{3}\partial_x^{-1}u\partial_x$, then $\mathcal{R}'_K(u)=-\frac{2}{3}-\frac{2}{3}\partial_x^{-1}(\cdot\partial_x)$.
\begin{proposition} \label{pr2.0}
Given $u\in H^{k+1}(\R)$ with $k\geq0$, there exists a unique linear operator
$$\mathcal{R}(u): H^{k+1}(\R)\rightarrow H^{k}(\R),$$
such that \eqref{eq:recursion2} and \eqref{eq:R Rstar} hold true. Moreover, $\mathcal{R}(u)$ is differentiable with respect to $u$.
\end{proposition}
\begin{proof} The idea is to relate the recursion operators $\mathcal{R}(u)$ and $\mathcal{R}_{12}$ \eqref{eq:recursion operator}. Suppose that $u\in\mathcal S(\R)$, then it reveals from \eqref{eq:BO hierarchy} and \eqref{eq:recursion2} that
\begin{eqnarray*}
&&\mathcal S(\R)\ni H'_{n+1}(u)=\frac{i}{2(n+1)}\mathcal{J}^{-1}\int_{\R}\delta(x_1-x_2)\mathcal{u}^{-}_{12}\mathcal{R}_{12}^{n+1}\cdot1\rmd x_2\\
&&=\mathcal{R}(u)H'_{n}(u)=\mathcal{R}(u)\frac{i}{2n}\mathcal{J}^{-1}\int_{\R}\delta(x_1-x_2)\mathcal{u}^{-}_{12}\mathcal{R}_{12}^{n}\cdot1\rmd x_2.
\end{eqnarray*}
The uniqueness of $\mathcal{R}(u)$ follows by an induction argument over $n$. Moreover, one infers that $\mathcal{R}(u)\sim-H\partial_x +L(u)$, where the higher order remainder term $L(u):\mathcal S(\R)\mapsto \mathcal S(\R)$ and which is differentiable. By a standard density argument, $\mathcal{R}(u)$ is also differentiable and $\mathcal{R}'(u)\sim L'(u)$.
\end{proof}

It will be shown in Section \ref{sec_4} that understanding the spectral information of the (adjoint) recursion operators $\mathcal{R}(u)$ and $\mathcal{R}^{\star}(u)$ is essential in proving the (spectral) stability of the BO multi-solitons.

We first observe that the differential equation \eqref{eq:stationary} verified by the soliton profile and the bi-Hamiltonian structure \eqref{eq:recursion2} imply that the $1$-soliton $Q_c(x-ct-x_0)$ with speed $c>0$ satisfies, for all $n\geq2$ and for any $t\in\R$, the following variational principle
\begin{eqnarray}\label{eq:1-sol variaprinciple}
&&H_{n+1}'(Q_c)+c H'_{n}(Q_c)=\mathcal{R}(Q_c)\big(H_{n}'(Q_c)+c H'_{n-1}(Q_c)\big)\nonumber\\&&=\cdot\cdot\cdot=\mathcal{R}^{n-1}(Q_c)\big(H_{2}'(Q_c)+c H'_{1}(Q_c)\big)=0,
\end{eqnarray}
\eqref{eq:1-sol variaprinciple} holds true since the functions $H_{n}'(Q_c)+c H'_{n-1}(Q_c)\in H^1(\R)$ which belongs to the domain of $\mathcal{R}(Q_c)$.
For future reference, we calculate here the quantities $H_j(Q_c)$  related to $1$-soliton profile $Q_c$. Instead of  applying the trace identity of $H_n$ \eqref{trace for H} directly, we multiply  \eqref{eq:1-sol variaprinciple} with $\frac{\rmd Q_c}{\rmd c}$, then for each $n$ one has
\[
\frac{\rmd H_{n+1}(Q_c)}{\rmd c}=-c\frac{\rmd H_{n}(Q_c)}{\rmd c}=\cdots=(-c)^n\frac{\rmd H_{1}(Q_c)}{\rmd c}=(-1)^n\pi c^{n},
\]
and therefore by inductions to have $\lim_{c\rightarrow0} H_{n}(Q_c)=0$ and
\begin{equation}\label{eq:jHamliton}
H_{n+1}(Q_c)=(-1)^n\frac{\pi}{n+1}c^{n+1}.
\end{equation}

Let us recall that the soliton  $Q_c(x-ct-x_0)$ \eqref{eq:so} is a solution of the BO equation. For simplicity, we denote $Q_c$ by $Q$. Then by \eqref{eq:recursion2}, we have
\begin{equation}\label{eq:recursion4}
 H'_{n+1}(Q)=\mathcal{R}(Q) H'_{n}(Q).
\end{equation}
To analyze the second variation of the actions,
we linearize the equation \eqref{eq:recursion2} to let $u=Q+\varepsilon z$, and obtain a relation between linearized operators $H''_{n+1}(Q)+cH''_{n}(Q)$ and $H''_{n}(Q)+cH''_{n-1}(Q)$ for all $n\geq2$. One has

\begin{proposition} \label{pr2.3} Suppose that $Q$ is a soliton profile of the BO equation with speed $c>0$, if $z\in H^{n}(\R)$ for $n\geq1$, then there holds
the following iterative operator identity
\begin{equation}\label{eq:recursion Hamiltonianre}
 \big(H''_{n+1}(Q)+cH''_{n}(Q)\big)z=\mathcal{R}(Q)\big(H''_{n}(Q)+cH''_{n-1}(Q)\big)z.
\end{equation}
\end{proposition}

\begin{proof} Let $u=Q+\varepsilon z$, by \eqref{eq:recursion2} and the definition of Gateaux derivative, one has
\begin{equation}\label{eq:recursion Hamiltonian2}
 H''_{n+1}(Q)z=\mathcal{R}(Q)(H''_{n}(Q)z)+\big(\mathcal{R}'(Q)z\big)(H_n'(Q)),
\end{equation}
then by \eqref{eq:recursion Hamiltonian2}
\[
\big(H''_{n+1}(Q)+cH''_{n}(Q)\big)z=\mathcal{R}(Q)\bigg(\big(H''_{n}(Q)+cH''_{n-1}(Q)\big)z\bigg)+\big(\mathcal{R}'(Q)z\big)\big(H_n'(Q)+cH_{n-1}'(Q)\big).
\]
 Notice that from Proposition \ref{pr2.0},  $\mathcal{R}'(Q)$ is well-defined, then \eqref{eq:recursion Hamiltonianre} follows directly from \eqref{eq:1-sol variaprinciple}.
\end{proof}

\section{Spectral Analysis} \label{sec_4}
\setcounter{equation}{0}

 Let $U^{ (m)}(t,x)$  be the BO $m$-solitons and $ U^{ (m)}(x)=U^{ (m)}(0,x)$ be the $m$-solitons profiles.  In this Section, we will use the subscript {\em{od}} to denote space of odd functions and the subscript {\em{ev}} to denote space of even functions.
A detailed spectral analysis of the linearized operator around $m$-solitons $\mathcal {L}_m$ (defined in  \eqref{eq:linearized n-soliton operator}) will be presented by employing the (adjoint) recursion operators defined in section \ref{sepecsec2}.

 The combination of two main
arguments allows to have the spectral information of $\mathcal {L}_m$. First, it was shown
that a form of iso-spectral property holds for linearized operators $\mathcal L_m$
around multi-solitons $U^{ (m)}(t,x)$, in the sense that the inertia (i.e. the number
of negative eigenvalues and the dimension of the kernel) is preserved along the
time evolution. Second, at large time, the linearized operator can be
viewed as a composition of several decoupled linearized operators
around each of the soliton profiles composing the multi-soliton, and
the spectrum of  linearized operator around the multi-solitons will converge to
the union of the spectra of the linearized operators around each
solitons.

More precisely, the linearized operators around the multi-solitons
fit in the framework of  Theorem 3 in \cite{LN}, we conclude that the inertia $in(\mathcal {L}_m(t))$ of $\mathcal {L}_m(t)$
is independent of $t$. Therefore, we can choose a convenient $t$ to calculate the inertia and
the best thing we can do is to calculate the inertia $in(\mathcal {L}_m(t))$ as $t$ goes to $\infty$. In particular,
the $m$-solitons $U^{ (m)}(t,x)$ splits into $m$ one-solitons $Q_{c_j}(x-c_jt-x_j)$ far apart \eqref{eq:asympotic behavior}. Then as $t$ goes to $\infty$, the spectrum $\sigma(\mathcal {L}_m(t))$ of $\mathcal {L}_m(t)$ converges
to the union of the spectrum $\sigma(L_{m,j})$ of $L_{m,j}:=I_m''(Q_{c_j})$. In this section, we show that the inertia of the linearized operator $\mathcal {L}_m$ related to the $m$-solitons $U^{ (m)}$ has exactly $[\frac{m+1}{2}]$ negative eigenvalues and the dimension of the null space equals to $m$, namely, $in(\mathcal {L}_m(t))=([\frac{m+1}{2}],m)$.
This result follows from an alternative inertia property of operators $L_{m,j}$:\\
--for $j=2k-1$ odd, $in(L_{m,j})=(1,1)$, i.e., $L_{m,2k-1}$ has exactly one negative eigenvalue;\\
--for $j=2k$ even, $in(L_{m,j})=(0,1)$, i.e., $L_{m,j2k}\geq0$ is positive.

In view of the expression of $L_{m,j}$, it is the summation of the operators $$H_{n+1}''(Q_{c_j})+c_jH_{n}''(Q_{c_j}) \quad \text{for}\ n=1,2,\cdot\cdot\cdot,m.$$ In particular, from Proposition \eqref{pr2.3}, it can be factorized in the following way
\begin{eqnarray}\label{formula of L}
L_{m,j}=\sum_{n=1}^m\sigma_{j,m-n}\bigg(H_{n+1}''(Q_{c_j})+c_jH_{n}''(Q_{c_j})\bigg)=\left(\prod_{k=1,k\neq j}^{m}(\mathcal R(Q_{c_j})+c_k)\right)\bigg(H_2''(Q_{c_j})+c_j H_1''(Q_{c_j})\bigg),
\end{eqnarray}
where $\sigma_{j,k}$ are the elementally symmetric functions of $c_1,c_2,\cdot\cdot\cdot,c_{j-1},c_{j+1},\cdot\cdot\cdot,c_m$ as follows,
$$\sigma_{j,0}=1, \ \sigma_{j,1}=\sum_{l=1,l\neq j}^mc_l, \ \sigma_{j,2}=\sum_{l<k,k,l\neq j}c_lc_k, ...,
\ \sigma_{j,m}=\prod_{l=1,l\neq j}^mc_l.$$

\subsection{The spectrum of $L_{1,c}$}
Let us deal with the linearized operator around one soliton profile $Q_{c}$, the associated linearized operator is,
\begin{equation}\label{eq:linearized operator}
\mathcal {L}_1=L_{1,c}=H_{2}''(Q_c)+cH_{1}''(Q_c)=-H\partial_x+c-2Q_c.
\end{equation}
 It is the purpose of this subsection to give an account of the spectral analysis for the operator $L_{1,c}$. We view $L_{1,c}$ as an unbounded, self-adjoint operator on $L^2(\R)$ with domain $H^1(\R)$, we refer to \cite{BBSDB83,HZ22} for some details of the following spectral analysis.

Using the fact that $Q_1$ decays to zero at infinity and Kato-Rellich's theorem, we know that the essential spectrum of $L_{1,1}$ is  $[1,+\infty)$. By differentiating \eqref{eq:stationary} with respect to $x_0$ and with respect to $c$, we obtain for normalized wave speed $c=1$,
\begin{equation}\label{eq:eigen kernel}
L_{1,1}Q'_1=0, \quad L_{1,1}(Q_1+xQ'_1)=-Q, \ \eta_0:=\frac{1}{\sqrt{\pi}}Q'_1,
\end{equation}
which show that $0$ is a discrete eigenvalue. It is inferred form \cite{BBSDB83}  that the other two discrete eigenvalues of $L_{1,1}$ and the associated normalized eigenfunctions are given by:
\begin{eqnarray}
&&\lambda_{-}=-\frac{1+ \sqrt{5}}{2}, \ \eta_{-}=\Lambda_-\big(2Q_1+(1+ \sqrt{5})Q_1^2\big),\  L_{1,1}\eta_{-}=\lambda_{-}\eta_{-},\label{ pos eigen}\\
&&\lambda_{+}=\frac{\sqrt{5}-1}{2}, \ \eta_{+}=\Lambda_+\big(2Q_1+(1- \sqrt{5})Q_1^2\big),\  L_{1,1}\eta_{+}=\lambda_{+}\eta_{+},\label{nega eigen}\\
&& \Lambda_{\pm}:=\frac{\big(1\pm\sqrt{5}\big)\big(\sqrt{5}\pm 2\big)^{\frac1 2}}{4 (\sqrt{5}\pi)^{\frac1 2}}.\nonumber
\end{eqnarray}
We can see that $1$ is also an eigenvalue. The corresponding eigenfunction is
\begin{equation}\label{eq:eigen 1}
\eta_1(x)=\frac{1}{\sqrt{\pi}}\big(Q'_1+xQ_1\big),\ L_{1,1}\eta_1=\eta_1.
\end{equation}

Now, we consider generalized eigenfunctions. For $\lambda>0$, let $\eta(x,\lambda)$ satisfy $L_{1,1}\eta=(\lambda+1)\eta$ with $\eta$ bounded as $x\rightarrow\pm\infty$. By a standard approach, we represent $\eta$ in the form
\begin{equation}\label{eq:eta repre}
\eta=\eta^{(+)}+\eta^{(-)},
\end{equation}
where $\eta^{(+)}(z)$ is analytic in the upper half complex plane and bounded as $\imp z\rightarrow+\infty$, whilst $\eta^{(-)}(z)$ is analytic in the lower half complex plane and bounded as $\imp z\rightarrow-\infty$. Since $L_{1,1}$ is real and the potential $Q_1(z)=Q_1^\ast(z^\ast)$, we can presume that
\begin{equation}\label{eq:eta decom}
\psi(z,\lambda)=\eta^{(+)}(z,\lambda)=\big(\eta^{(-)}(z^\ast,\lambda)\big)^\ast.
\end{equation}
By \eqref{Hul} and substituting  \eqref{eq:eta repre} into $L_1\eta=(\lambda+1)\eta$, we have
\[
i\eta^{(+)}_z-i\eta^{(-)}_z+\big(2Q_1(z)+\lambda\big)\big(\eta^{(+)}+\eta^{(-)}\big)=0,
\]
which by \eqref{eq:eta decom} is equivalent to
\[
i\psi_z+\big(2Q_1(z)+\lambda\big)\psi=0,
\]
the solution of which is
\[\psi(z)=\frac{1}{\sqrt{2\pi}}\frac{z-i}{z+i}e^{i\lambda z}.\]
The generalized eigenfunctions of $L_{1,1}$ is thus given by \eqref{eq:eta repre} and \eqref{eq:eta decom}, the explicit formula is
\[
\eta(x,\lambda)=\sqrt{\frac{2}{\pi}}\frac{\big(x^2-1\big)\cos(\lambda x)+2x \sin(\lambda x)}{x^2+1}.
\]

For $j,k\in \sigma:=\{-,0,+,1\}$, the associated four functions $\eta_\sigma (x)$ defined in \eqref{eq:eigen kernel}, \eqref{ pos eigen}, \eqref{nega eigen} and \eqref{eq:eigen 1}, combining with the generalized eigenfunctions $\psi(x,\lambda)$ \eqref{eq:eta decom}, there holds the following $L^2$-inner product properties:
\begin{eqnarray}\label{completeness in L22}
&&\langle \eta_j,\eta_k\rangle=\delta_{jk},\nonumber\\
&&\langle\psi(\cdot,\lambda),\psi^\ast(\cdot,\lambda')\rangle=\delta(\lambda-\lambda'),\ \ \langle\psi(\cdot,\lambda),\eta_j\rangle=0,\nonumber\\
&&\int_0^{+\infty}\bigg(\psi(x,\lambda)\psi^\ast(y,\lambda)+\psi^\ast(x,\lambda)\psi(y,\lambda)\bigg)\rmd \lambda+\sum_{j\in\sigma}\eta_j(x)\eta_j(y)=\delta(x-y).\label{completeness in L2}
\end{eqnarray}
\eqref{completeness in L2} means the completeness of the implied eigenfunction expansion in $L^2(\R)$. In particular, for any function $f\in L^2(\R)$, one can decompose which into the above basis as follows:
\begin{eqnarray} \label{decomposition of z1}
&& f(x)=\int_0^{+\infty}\bigg(\tilde{\alpha}(\lambda)\psi(x,\lambda)+\tilde{\alpha}^\ast(\lambda)\psi^\ast(x,\lambda)\bigg)\rmd \lambda +\tilde{\alpha}_j\eta_j(x) ,\\
&& \tilde{\alpha}(\lambda):=\langle f,\psi^\ast(\lambda)\rangle,\ \ \tilde{\alpha}_j:=\langle f,\eta_j(\lambda)\rangle,\  j\in \sigma=\{-,0,+,1\}.\nonumber
\end{eqnarray}

To obtain the spectrum of the operator $L_{m,j}$ \eqref{formula of L}, let us consider the spectral analysis of the linearized operators
\begin{eqnarray}\label{formula of Ln}
L_n:=H''_{n+1}(Q)+cH''_{n}(Q),
\end{eqnarray} for all integers $n\geq1$. Here we write for simplicity $Q_c$ by $Q$ in the rest of this section. It is nature to consider the quadratic form $\langle L_nz,z\rangle$ with the decomposition of $z(x)$ in \eqref{decomposition of z1}. However, it is quite involved as the eigenfunctions of the operator $L_1=L_{1,c}$ \eqref{eq:linearized operator}  need not to be the eigenfunctions of $L_n$ for $n\geq2$. Our main ingredient part of the spectral analysis of $L_n$ is the observation that $JL_n$ share the same eigenfunctions of $JL_1$. To deal with this spectrum problem, the core is the following operator identities related to the recursion operator $\mathcal{R}(Q)$ and the adjoint recursion operator $\mathcal{R}^{\star}(Q)$ (see \eqref{eq:adjoint of R}).

\begin{lemma} \label{le3.5} The recursion operator $\mathcal{R}(Q)$, the adjoint recursion operator $\mathcal{R}^{\star}(Q)$ and the linearized operator $L_n$ for all integers $n\geq1$ satisfy the following operator identities.
\begin{eqnarray}
&&L_n\mathcal{J}\mathcal{R}(Q)=\mathcal{R}(Q)L_n\mathcal{J}, \label{operator identity1}\\
&&\mathcal{J}L_n\mathcal{R}^{\star}(Q)=\mathcal{R}^{\star}(Q)\mathcal{J}L_n,\label{operator identity2}
\end{eqnarray}
where $\mathcal{J}$ is the operator $\partial_x$.
\end{lemma}
\begin{proof}
We need only to prove \eqref{operator identity2}, since one takes the adjoint operation on \eqref{operator identity2} to have \eqref{operator identity1}. Notice that from
Proposition \ref{pr2.3}, one has that the operator $\mathcal{R}(Q)L_n=L_{n+1}$ is self-adjoint. This in turn implies that $$(\mathcal{R}(Q)L_n)^\star=\mathcal{R}(Q)L_n=L_n\mathcal{R}^\star(Q),$$
On the other hand, in view of \eqref{eq:R Rstar}, one has
$$
\mathcal{J}L_n\mathcal{R}^\star(Q)=\mathcal{J}\mathcal{R}(Q)L_n=\mathcal{R}^\star(Q)\mathcal{J}L_n,
$$
as the advertised result in the lemma.
\end{proof}

\begin{remark}
Types of \eqref{operator identity1} and \eqref{operator identity2} hold for any solutions of the BO equation. In particular, let $U^{ (m)}$ be the BO $m$-soliton profile and $\mathcal{L}_m$ be the second variation operator defined in \eqref{eq:linearized n-soliton operator}. Then it is easy to verify that (similar to  Lemma \ref{le3.5}) the following operator identities hold true
\begin{eqnarray}
&&\mathcal{L}_m\mathcal{J}\mathcal{R}(U^{ (m)})=\mathcal{R}(U^{ (m)})\mathcal{L}_m\mathcal{J},\label{operator identity3}\\
&&\mathcal{J}\mathcal{L}_m\mathcal{R}^{\star}(U^{ (m)})=\mathcal{R}^{\star}(U^{ (m)})\mathcal{J}\mathcal{L}_m.\label{operator identity4}
\end{eqnarray}
\end{remark}

An immediate consequence of the factorization results \eqref{operator identity1} and \eqref{operator identity2} is that the (adjoint) recursion operator $\mathcal{R}(Q)$($\mathcal{R}^\star(Q)$) and $L_n\mathcal{J}$ ($\mathcal{J}L_n$) are commutable.  It then turns out  that the operators $\mathcal{J}L_n$ and $\mathcal{R}^\star(Q)$ share the same eigenfunctions, and  $L_n\mathcal{J}$ shares the same eigenfunctions with the recursion operator $\mathcal{R}(Q)$. It will be possible to derive the precise eigenvalues of operators $L_n\mathcal{J}$ and $\mathcal{J}L_n$ by analyzing the asymptotic behaviors of the corresponding eigenfunctions.

Our approach for the spectral analysis of the linearized operator $L_n$ is as follows.  Firstly, we derive the spectrum of the operator $\mathcal{J}L_n$, which is more easier than to have the spectrum of $L_n$. The idea is motivated by \eqref{operator identity2} to reduce to the spectrum of the  adjoint recursion operator $\mathcal{R}^\star(Q)$. We then  show that the eigenfunctions of $\mathcal{R}^\star(Q)$ ($\mathcal{J}L_n$) plus a generalized kernel of $\mathcal{J}L_n$ form an orthogonal basis in $L^2(\R)$, which can be viewed as a completeness relation. Finally, we calculate the quadratic form $\langle L_nz,z\rangle$ with function $z$ has a decomposition in the above basis, and the inertia of $L_n$ can be computed directly.

\subsection{The spectrum of the recursion operator around the BO one soliton}

The spectrum of the recursion operator $\mathcal{R}(Q)$ and its adjoint operator $\mathcal{R}^\star(Q)$ are essential to analyze the linearized operator $L_n$ defined in \eqref{formula of Ln}. Note that the recursion operators are nonlocal and  even not explicit, which are major obstacles to study them directly. However, by employing the properties of the squared eigenfunctions of the eigenvalue problem \eqref{space lax}, one could have the following result.

\begin{lemma} \label{le3.6} The recursion operator $\mathcal{R}(Q)$ defined in $L^2(\R)$ with domain $H^{1}(\R)$  has only one discrete eigenvalue $-c$  associated with the eigenfunction $Q$, the essential spectrum is the interval $[0,+\infty)$, and the corresponding eigenfunctions do not have spatial decay and not in $L^2(\R)$. Moreover, the kernel of $\mathcal{R}(Q)$ is spanned by $\big(N\bar{N}^\ast\big)(x,0)$ where $N(x,\lambda)$ and  $\bar{N}(x,\lambda)$ are defined in \eqref{phi3} and \eqref{phi2}.
\end{lemma}

\begin{proof} Consider the Jost solutions of the spectral problem \eqref{space lax} with the potential $u=Q$ and the asymptotic expressions in \eqref{Jost be1}, \eqref{Jost be2}, \eqref{2.N}, \eqref{2.barN} and \eqref{2.M}. In this case, \eqref{space lax} possesses only one discrete eigenvalue $\lambda_1=-\frac{c}{2}<0$ which generates the soliton profile $Q$. The key ingredient in the analysis is to find the eigenvalues of $\mathcal{R}_{12}(Q)$ in \eqref{eq:recursion operator} around the soliton profile $Q$, as $\mathcal{R}(Q)$ is not explicit. It is then found that (using the properties of the generalized Hilbert transform presented in Subsection \ref{sec_2.0} and $Q_{12}^-Q_{12}^+=Q_{12}^+Q_{12}^-$) for $\lambda>0$, there holds the following
 \begin{eqnarray}
&&\bigg(Q_{12}^+-iQ_{12}^-H_{12}\bigg)\bigg(Q_{12}^-(N(x_1,\lambda)\bar{N}^\ast(x_2,\lambda))\bigg)=-4\lambda Q_{12}^-(N(x_1,\lambda)\bar{N}^\ast(x_2,\lambda)),\label{eq:eigen ident1}\\
&&\bigg(Q_{12}^+-iQ_{12}^-H_{12}\bigg)\bigg(Q_{12}^-(N^\ast(x_1,\lambda)\bar{N}(x_2,\lambda))\bigg)=-4\lambda Q_{12}^-(N^\ast(x_1,\lambda)\bar{N}(x_2,\lambda)),\label{eq:eigen ident2}\\
&&\bigg(Q_{12}^+-iQ_{12}^-H_{12}\bigg)\bigg(Q_{12}^-(\Phi_1(x_1)\Phi_1^\ast(x_2))\bigg)=-4\lambda Q_{12}^-(\Phi_1(x_1)\Phi_1^\ast(x_2)),\label{eq:eigen ident3}
\end{eqnarray}
where $\bar{N}^\ast(x_2),\Phi_1^\ast(x_2)$ satisfy the adjoint eigenvalue problem of \eqref{space lax} with  potential $u=Q$ ( i.e., replace $i,x$  by $-i,x_2$ in\eqref{space lax}). Recall that $Q_{12}^{\pm}=Q(x)\pm Q(x_2)+i(\partial_{x}\mp\partial_{x_2})$ defined similarly as in \eqref{eq:functions}. Then  \eqref{eq:eigen ident1},\eqref{eq:eigen ident2} and \eqref{eq:eigen ident3} reveal that
\begin{eqnarray}
&&\mathcal{R}_{12}(Q)\big(N(x_1,\lambda)\bar{N}^\ast(x_2,\lambda)\big)=4\lambda\big(N(x_1,\lambda)\bar{N}^\ast(x_2,\lambda)\big),\label{eq:eigen ident4}\\
&&\mathcal{R}_{12}(Q)\big(N^\ast(x_1,\lambda)\bar{N}(x_2,\lambda)\big)=4\lambda\big(N^\ast(x_1,\lambda)\bar{N}(x_2,\lambda)\big),\label{eq:eigen ident5}\\
&&\mathcal{R}_{12}(Q)\big(\Phi_1(x_1)\Phi_1^\ast(x_2)\big)=4\lambda_1\big(\Phi_1(x_1)\Phi_1^\ast(x_2)\big)=-2c\big(\Phi_1(x_1)\Phi_1^\ast(x_2)\big).\label{eq:eigen ident6}
\end{eqnarray}
 In view of the extra factor $\frac12$ in the bi-Hamiltonian structure \eqref{eq:BO hierarchy}, one sees that
 the squared eigenfunctions $N\bar{N}^\ast$, $N^\ast\bar{N}$ satisfy
\begin{eqnarray}
&&\mathcal{R}(Q)\big(N\bar{N}^\ast\big)(x,\lambda)=2\lambda\big(N\bar{N}^\ast\big)(x,\lambda),\quad \text{for} \ \lambda>0, \label{eq:Reigen} \\ &&\mathcal{R}(Q)\big(N^\ast\bar{N}\big)(x,\lambda)=2\lambda\big(N^\ast\bar{N}\big)(x,\lambda),\quad \text{for} \ \lambda>0, \label{eq:Reigen0} \\
&&\mathcal{R}(Q)\big(\Phi_1\Phi_1^\ast\big)(x)=2\lambda_1\big(\Phi_1\Phi_1^\ast\big)(x)=-c\big(\Phi_1\Phi_1^\ast\big)(x).\label{eq:R eigen1}
\end{eqnarray}
\eqref{eq:R eigen1}  and $\Phi_1\Phi_1^\ast=\frac{c}{2}Q$ reveal that $\mathcal{R}(Q)Q=-cQ$. Moreover, if we differentiate \eqref{eq:R eigen1} with respect to $c$, it follows that there holds
\begin{equation*}\label{eq:generalized kernel}
\mathcal{R}(Q)\frac{\partial Q}{\partial c}=-Q-c\frac{\partial Q}{\partial c}.
\end{equation*}
 On account of  \eqref{eq:Reigen} and \eqref{eq:Reigen0}, the essential spectrum of $\mathcal{R}(Q)$ is given by  $2\lambda\geq0$, which equals to the interval $[0,+\infty)$.
The associated generalized eigenfunctions $\big(N\bar{N}^\ast\big)(x,\lambda)$ and $\big(N^\ast\bar{N}\big)(x,\lambda)$ possess no spatial decay and not in $L^2(\R)$ which can be seen from  \eqref{phi2} and \eqref{phi3}.

On the other hand, a simple direct computation shows that the kernel of $\mathcal{R}(Q)$ is reached at $\lambda=0$, in view of \eqref{phi1}, \eqref{phi2} and \eqref{phi3}, the associated eigenfunction is  $$\big(N\bar{N}^\ast\big)(x,0)=|N(x,0)|^2\notin L^2(\R).$$ The proof of the lemma is completed.
\end{proof}

Similar to the proof of Lemma \ref{le3.6}, we have the following result concerning the spectrum of  the composite operators  $\mathcal{R}^n(Q)$ for $n\geq2$.

\begin{corollary} \label{le3.8} The composite operator $\mathcal{R}^n(Q)$ defined in $L^2(\R)$ with domain $H^{n}(\R)$ has only one eigenvalue $(-c)^n$  associated with the eigenfunction $Q$, the essential spectrum is the interval $[0,+\infty)$, and the corresponding generalized eigenfunctions do not have spatial decay and not in $L^2(\R)$.
\end{corollary}

We now consider the adjoint recursion operator $\mathcal{R}^\star(Q)$.  In view of the factorization \eqref{operator identity2}, it shares the same eigenfunctions of $\mathcal{J}L_n$ and thus is more relevant to the spectral stability problems of solitons. Recall that \eqref{eq:adjoint of R} implies
$$
\mathcal{R}^{\star}(u)=J\mathcal{R}(u)J^{-1}.
$$
The spectral information of $\mathcal{R}^{\star}(Q)$ can be derived as follows.

\begin{lemma} \label{le3.81} The adjoint recursion operator $\mathcal{R}^\star(Q)$ defined in $L^2(\R)$ with domain $H^{1}(\R)$  has only one eigenvalue $-c$ associated with the eigenfunction $Q_x$, the essential spectrum is the interval $[0,+\infty)$, and the corresponding eigenfunctions do not have spatial decay and not in $L^2(\R)$. Moreover, the kernel of $\mathcal{R}^\star(Q)$ is spanned by $\big(N\bar{N}^\ast\big)_x(x,0)$.
\end{lemma}
\begin{proof}
Consider the Jost solutions  of the spectral problem \eqref{space lax} with the potential $Q$ and the asymptotic formulas in  \eqref{Jost be1}, \eqref{Jost be2}, \eqref{2.N}, \eqref{2.barN} and \eqref{2.M}. The soliton profile $Q$ is  generated by the eigenvalue $\lambda_1=-\frac{c}{2}$. Similar to the proof of Lemma \ref{le3.6}, we find the eigenvalue of $\mathcal{R}^\star_{12}(Q)$ in \eqref{eq:recursion operator} around the soliton profile $Q$, as $\mathcal{R}^\star(Q)$ is not explicit. It is then found from \eqref{eq:eigen ident1}, \eqref{eq:eigen ident2} and \eqref{eq:eigen ident3}  that for $\lambda>0$, one has
\begin{eqnarray*}
&&\mathcal{R}^\star_{12}(Q)\big(Q_{12}^-N(x_1)\bar{N}^\ast(x_2)\big)=4\lambda\big(Q_{12}^-N(x_1)\bar{N}^\ast(x_2)\big),\label{eq:eigen ident7}\\
&&\mathcal{R}^\star_{12}(Q)\big(Q_{12}^-N^\ast(x_1)\bar{N}(x_2)\big)=4\lambda\big(Q_{12}^-N^\ast(x_1)\bar{N}(x_2)\big),\label{eq:eigen ident8}\\
&&\mathcal{R}^\star_{12}(Q)\big(Q_{12}^-\Phi_1(x_1)\Phi_1^\ast(x_2)\big)=4\lambda_1\big(Q_{12}^-\Phi_1(x_1)\Phi_1^\ast(x_2)\big)=-2c\big(Q_{12}^-\Phi_1(x_1)\Phi_1^\ast(x_2)\big).\label{eq:eigen ident9}
\end{eqnarray*}
As a consequence, there holds the following relations
\begin{eqnarray}
&&\mathcal{R}^\star(Q)\big(N\bar{N}^\ast\big)_x(x,\lambda)=2\lambda\big(N\bar{N}^\ast\big)_x(x,\lambda),\quad \text{for} \ \lambda>0,\label{eq:Reigen2} \\ &&\mathcal{R}^\star(Q)\big(N^\ast\bar{N}\big)_x(x,\lambda)=2\lambda\big(N^\ast\bar{N}\big)_x(x,\lambda),\quad \text{for} \ \lambda>0,\label{eq:Reigen21} \\
&&\mathcal{R}^\star(Q)\big(\Phi_1\Phi_1^\ast\big)_x(x)=2\lambda_1\big(\Phi_1\Phi_1^\ast\big)_x(x)=-c\big(\Phi_1\Phi_1^\ast\big)_x(x),\label{eq:Reigen3}\\
&&\mathcal{R}^\star(Q)\frac{\partial Q_x}{\partial c}=-Q_x-c\frac{\partial Q_x}{\partial c}.\label{eq:Reigen4}
\end{eqnarray}
Since by \eqref{eq:Reigen3}, one has $\mathcal{R}^\star(Q)Q_x=-cQ_x$, then one sees that $-c$ is the only discrete eigenvalue. In view of \eqref{eq:Reigen2} and \eqref{eq:Reigen21}, the essential spectrum of $\mathcal{R}^\star(Q)$ is  $2\lambda\geq0$ which is the interval $[0,+\infty)$.
The associated generalized eigenfunctions $\big(N\bar{N}^\ast\big)_x(x,\lambda)$ possess no spatial decay and not in $L^2(\R)$ which can be seen from   \eqref{phi2} and \eqref{phi3}.

Similarly, the kernel of $\mathcal{R}^\star(Q)$ is attached at $\lambda=0$ and the associated kernel is $\big(N\bar{N}^\ast\big)_x(x,0)$. This completes the proof of Lemma \ref{le3.81}.
\end{proof}

\begin{remark}\label{4.0}
The spectral information of  $\mathcal{R}(Q)$ presented in  Lemma \ref{le3.6} and $\mathcal{R}^\star(Q)$ in  Lemma \ref{le3.81} reveal that $\mathcal{R}(Q)$ and $\mathcal{R}^\star(Q)$ are essentially invertible in $L^2(\R)$.
\end{remark}

\subsection{The spectrum of linearized operators $\mathcal{J}L_n$, $L_n\mathcal{J}$ and $L_n$}
In this subsection our attention is focused on  the spectral analysis of the  linearized operators $\mathcal{J}L_n$, $L_n\mathcal{J}$ and $L_n$. The main ingredients are \eqref{operator identity2} the observation that the eigenfunctions of the adjoint recursion operator $\mathcal{J}L_n$ and its generalized eigenfunction form an orthogonal basis in $L^2(\R)$ (see \eqref{decomposition of z} below). It follows that the spectra of $\mathcal{J}L_n$ lies on the imaginary axis which implies directly the spectral stability of the BO solitons.

Let us first deal with the $n=1$ case, recall from \eqref{Jost be1} that $|(N\bar{N}^\ast)(x,\lambda)- e^{i\lambda x}|\rightarrow0$ as $x\rightarrow +\infty$, then we can summarize the spectral information of $\mathcal{J}L_1$  as follows:
\begin{eqnarray*}
&&\mathcal{J}L_1\big(N\bar{N}^\ast\big)_x=i(\lambda^2+\lambda)\big(N\bar{N}^\ast\big)_x,\quad \text{for} \ \lambda>0,\label{eq:BOessentia} \\
&&\mathcal{J}L_1\big(N^\ast\bar{N}\big)_x=-i(\lambda^2+\lambda)\big(N^\ast\bar{N}\big)_x,\quad \text{for} \ \lambda>0,\label{eq:BOessentia'} \\
&&\mathcal{J}L_1\big(\Phi_1\Phi_1^{\ast}\big)_x=\frac{c}{2}\mathcal{J}L_1Q_x=0,\label{eq:BOkernel}\\
&&\mathcal{J}L_1\frac{\partial Q}{\partial c}=-Q_x.\label{eq:BOgenera kernel}
\end{eqnarray*}

Similarly, key spectral information of the operator $L_1\mathcal{J}$ is the following
\begin{eqnarray*}
&&L_1\mathcal{J}\big(N\bar{N}^\ast\big)=i(\lambda^2+\lambda)\big(N\bar{N}^\ast\big), \quad \text{for} \ \lambda>0 , \label{eq:BOessentia2} \\
&&L_1\mathcal{J}\big(N^\ast\bar{N}\big)=-i(\lambda^2+\lambda)\big(N^\ast\bar{N}\big), \quad \text{for} \ \lambda>0 ,; \label{eq:BOessentia2'} \\
&&L_1\mathcal{J}\big(\Phi_1\Phi_1^{\ast}\big)=\frac{c}{2}L_1Q_x=0,\label{eq:BOkerne2}\\
&&L_1\mathcal{J}\partial_x^{-1}\frac{\partial Q}{\partial c}= L_1\frac{\partial Q}{\partial c}=-Q.\label{eq:BOgenera kerne2}
\end{eqnarray*}
Here the function $\partial_x^{-1}\frac{\partial Q}{\partial c}\in L^2(\R)$ is well defined since  $\frac{\partial Q}{\partial c}=\frac{2(1-c^2x^2)}{(c^2x^2+1)^2)}\in H^1(\R)$.
The eigenfunctions presented above in terms of the squared eigenfunctions of the eigenvalue problem of the BO equation \eqref{space lax} with the potential $u=Q$. In this case, $\beta(\lambda)=0$ for $\lambda>0$ and there exists only one discrete eigenvalue $\lambda_1=-\frac{c}{2}$, the Jost solutions are explicitly given by \eqref{phi1}, \eqref{phi2} and \eqref{phi3}. The squared eigenfunctions generate the two function sets as follows. The first set
\begin{equation}\label{eq:set JL}
\{\big(N\bar{N}^\ast\big)_x(x,\lambda),\ \big(N^\ast\bar{N}\big)_x(x,\lambda) \quad \text{for} \quad \lambda>0;\quad Q_x; \quad \frac{\partial Q}{\partial c}  \}
\end{equation}
consists of  linearly independent eigenfunctions and generalized kernel of the operator $\mathcal{J}L_1$. Moreover, they are essentially orthogonal under the $L^2$-inner product.
The second set
\begin{equation}\label{eq:set LJ}
\{\big(N\bar{N}^\ast\big)(x,\lambda),\ \big(N^\ast\bar{N}\big)(x,\lambda) \quad \text{for} \quad \lambda>0;\quad Q; \quad \partial_x^{-1}\frac{\partial Q}{\partial c} \}
\end{equation}
consists of linearly independent eigenfunctions and generalized kernel of the operator $L_1\mathcal{J}$. Notice that the function $\frac{\partial Q}{\partial c}$ is even, by using the asymptotic behaviors of the Jost solutions in \eqref{Jost be1}, \eqref{Jost be2}, \eqref{2.N}, \eqref{2.barN} and \eqref{2.M}, for $\lambda,\lambda'>0$, one can compute the inner product of the elements of the sets \eqref{eq:set JL} and \eqref{eq:set LJ} as the following (see \cite{KLM99}):
\begin{eqnarray}
&&\int_{\R}\big(N\bar{N}^\ast\big)_x(x,\lambda)(N^\ast\bar{N}\big)(x,\lambda')\rmd x=-2\pi i\lambda\delta(\lambda-\lambda'),\label{eq:product1} \\
&&\int_{\R}\big(N^\ast\bar{N}\big)_x(x,\lambda)(N\bar{N}^\ast\big)(x,\lambda')\rmd x=2\pi i\lambda\delta(\lambda-\lambda'),\label{eq:product1'} \\
&&\int_{\R}\big(N\bar{N}^\ast\big)_x(x,\lambda)(N\bar{N}^\ast\big)(x,\lambda')\rmd x=\int_{\R}\big(N^\ast\bar{N}\big)_x(x,\lambda)(N^\ast\bar{N}\big)(x,\lambda')\rmd x=0,\label{eq:product1''} \\
&&\int_{\R}Q_x\partial_x^{-1}\big(\frac{\partial Q}{\partial c}\big)\rmd x=-\int_{\R}Q\frac{\partial Q}{\partial c}\rmd x=-\frac{\rmd H_1(Q)}{\rmd c}=-\pi,\label{eq:product2}\\
&&\int_{\R}\frac{\partial Q}{\partial c}Q\rmd x=\frac{\rmd H_1(Q)}{\rmd c}=\pi.\label{eq:product3}
\end{eqnarray}
The corresponding closure or completeness relation is
\begin{eqnarray}\label{eq:closure relation}
&&\frac{1}{2\pi i}\int_0^{+\infty}\bigg(\big(N\bar{N}^\ast\big)_x(x,\lambda)\big(N^\ast\bar{N}\big)(y,\lambda)-\big(N^\ast\bar{N}\big)_x(x,\lambda)\big(N\bar{N}^\ast\big)(y,\lambda)\bigg)\frac{\rmd \lambda}{\lambda}\nonumber\\
&&+\frac{1}{\pi}\bigg(Q(y)\frac{\partial Q(x)}{\partial c}-Q_x\partial_y^{-1}\frac{\partial Q(y)}{\partial c}\bigg)=\delta(x-y),
\end{eqnarray}
which indicates that any function $z(y)$ which vanishes at $x\rightarrow\pm\infty$  can be expanded over the above two bases \eqref{eq:set JL} and \eqref{eq:set LJ}. In particular, we have the following decomposition of the function $z$:
\begin{eqnarray} \label{decomposition of z}
&& z(x)=\int_0^{+\infty}\bigg(\alpha(\lambda)\big(N\bar{N}^\ast\big)_x(x,\lambda)+\alpha^\ast(\lambda)\big(N^\ast\bar{N}\big)_x(x,\lambda)\bigg)\rmd \lambda+\beta Q_x+\gamma \frac{\partial Q}{\partial c},\\
&&\alpha(\lambda)=\frac{1}{2\pi i\lambda}\langle\big(N^\ast\bar{N}\big)(y,\lambda),z(y)\rangle,\ \beta= \frac{1}{\pi}\langle\partial_y^{-1}\frac{\partial Q(y)}{\partial c},z(y)\rangle,\ \gamma=\frac{1}{\pi}\langle Q(y),z(y)\rangle.
\end{eqnarray}
Similarly, one can also decompose the function $z(x)$ on the second set \eqref{eq:set LJ} by multiplying \eqref{eq:closure relation} with $z(x)$ and integrating with $\rmd x$.

We now  consider the operator $\mathcal{J}L_n$. Since $L_n=H''_{n+1}(Q)+cH''_n(Q)$ given by \eqref{formula of Ln} which is defined in $L^2(\R)$ with domain $H^{n}(\R)$,  the symbol of the principle (constant coefficient) part of which is $$\big(H''_{n+1}(0)+cH''_n(0)\big)^{\wedge}(\xi)=\frac{2^n}{n+1}\widehat{(-H\partial_x)^n}+\frac{2^{n-1}c}{n}\widehat{(-H\partial_x)^{n-1}}=\frac{2^n}{n+1}|\xi|^n+\frac{2^{n-1}c}{n}|\xi|^{n-1},$$
it thus transpires that the symbol of the principle part of the operator $\mathcal{J}L_n$ is
\begin{equation}\label{eq:symbol JL}
\varrho_{n,c}(\xi):=i\frac{2^n}{n+1}|\xi|^{n}\xi+i\frac{2^{n-1}c}{n}|\xi|^{n-1}\xi.
\end{equation}

We have the following statement which concerning the spectrum for the operator $\mathcal{J}L_n$.
\begin{proposition} \label{pr3.5-3}
The essential spectra of  $\mathcal{J}L_n$ (defined in $L^2(\R)$ with domain $H^{n+1}(\R)$) for $n\geq1$ is  $i\R$, the kernel is spanned by the function $Q_x$ and the generalized kernel is spanned by $\frac{\partial Q}{\partial c}$.
\end{proposition}
\begin{proof}
 The proof is by direct verification. We compute the spectrum of the operator $\mathcal{J}L_n$ directly by employing the squared eigenfunctions as follows
\begin{eqnarray}
&&\mathcal{J}L_n\big(N\bar{N}^\ast\big)_x=\varrho_{n,c}(\lambda)\big(N\bar{N}^\ast\big)_x,\quad \text{for} \ \lambda>0,\label{eq:essentia} \\
&&\mathcal{J}L_n\big(N^\ast\bar{N}\big)_x=\varrho^\ast_{n,c}(\lambda)\big(N^\ast\bar{N}\big)_x,\quad \text{for} \ \lambda>0,\label{eq:essentia'}\\
&&\mathcal{J}L_n\big(\Phi_1\Phi_1^{\ast}\big)_x=\frac{c}{2}\mathcal{J}L_nQ_x=0,\label{eq:kernel}\\
&&\mathcal{J}L_n\frac{\partial Q}{\partial c}=(-1)^{n}c^{n-1}Q_x.\label{eq:genera kernel}
\end{eqnarray}
In view of \eqref{eq:symbol JL}, \eqref{eq:essentia} and \eqref{eq:essentia'},  the essential spectrum of $\mathcal{J}L_n$ are  $\pm\varrho_{n,c}(\lambda)$ for $ \lambda>0$, which is the whole imaginary axis. In view of \eqref{eq:kernel} and \eqref{eq:genera kernel}, the kernel and generalized kernel of $\mathcal{J}L_n$ is $Q_x$ and $\frac{\partial Q}{\partial c}$, respectively.  The proof of Proposition \ref {pr3.5-3} is completed.
\end{proof}

For the adjoint operator of $\mathcal{J}L_n$, namely, the operator $-L_n\mathcal{J}$, for the spectrum of which, we have the following result.

\begin{proposition} \label{le3.82}
The essential spectrum of  $L_n\mathcal{J}$  (defined in $L^2(\R)$ with domain $H^{n+1}(\R)$) for $n\geq1$ is  $i\R$, the kernel is spanned by the function $Q$ and the generalized kernel is spanned by $\partial_x^{-1}\big(\frac{\partial Q}{\partial c}\big)$.
\end{proposition}
\begin{proof}
One can compute the spectrum of the operator $L_n\mathcal{J}$ directly by employing the squared eigenfunctions as follows
\begin{eqnarray}
&&L_n\mathcal{J}\big( N\bar{N}^\ast\big)=L_n\big(N\bar{N}^\ast\big)_x=\varrho_{n,c}(\lambda) N\bar{N}^\ast, \quad \text{for} \ \lambda>0;\label{eq:essentia2} \\
&&L_n\mathcal{J}\big( N^\ast\bar{N}\big)=L_n\big(N^\ast\bar{N}\big)_x=\varrho^\ast_{n,c}(\lambda)N^\ast\bar{N},\quad \text{for} \ \lambda>0, \label{eq:essentia2'} \\
&&L_n\mathcal{J}\Phi_1\Phi_1^{\ast}=\frac{c}{2} L_nQ_x=0,\label{eq:kerne2}\\
&&L_n\mathcal{J}\partial_x^{-1}\big(\frac{\partial Q}{\partial c}\big)=L_n\big(\frac{\partial Q}{\partial c}\big)=(-1)^{n}c^{n-1}Q.\label{eq:genera kerne2}
\end{eqnarray}
In view of \eqref{eq:symbol JL}, \eqref{eq:essentia2} and \eqref{eq:essentia2'}, the essential spectrum of $L_n\mathcal{J}$ is $\pm\varrho_{n,c}(\lambda)$ for $ \lambda>0$ which is  the whole imaginary axis. In view of \eqref{eq:kerne2} and \eqref{eq:genera kerne2}, the kernel and generalized kernel of $\mathcal{J}L_n$ is $Q$ and $\partial_x^{-1}\frac{\partial Q}{\partial c}$, respectively. The proof is concluded.
\end{proof}

With the decomposition of function $z(x)$ in \eqref{decomposition of z}, we can compute the quadratic form related to the operator  $L_n$ and illustrate the spectral information.
The following statement describes  the full spectrum of linearized operator $L_n=H''_{n+1}(Q)+cH''_n(Q)$ for $n\geq1$.

\begin{lemma} \label{le3.9} For $n\geq1$ and any $z\in H^{\frac{n}{2}}_{od}(\R)$, we have $\langle L_{n}z,z\rangle\geq0$ and $\langle L_{n}z,z\rangle=0$ if and only if $z$ is a multiple of $Q_x$. In $H^{\frac{n}{2}}_{ev}(\R)$ and for odd $n$, the operator $L_{n}$ has exactly one negative eigenvalue and zero is not an eigenvalue any more; In $H^{\frac{n}{2}}_{ev}(\R)$ and for $n$ even, the operator $L_{n}$ has no negative eigenvalue.
\end{lemma}
\begin{proof}
For any $z(x)\in H^{\frac{n}{2}}(\R)$, we have the decomposition \eqref{decomposition of z}, then we can evaluate the quadratic form $\langle L_nz,z \rangle$ as follows,
\begin{eqnarray} \label{quadratic}
 &&\langle L_nz,z \rangle=\langle\int_0^{+\infty}\bigg(\alpha(\lambda)L_n\big(N\bar{N}^\ast\big)_x(x,\lambda)+\alpha^\ast(\lambda)L_n\big(N^\ast\bar{N}\big)_x(x,\lambda)\bigg)\rmd \lambda,\nonumber\\&& \int_0^{+\infty}\bigg(\alpha(\lambda)\big(N\bar{N}^\ast\big)_x(x,\lambda)+\alpha^\ast(\lambda)\big(N^\ast\bar{N}\big)_x(x,\lambda)\bigg)^\ast\rmd \lambda\rangle\nonumber\\&&+2\gamma\langle \int_0^{+\infty}\bigg(\alpha(\lambda)L_n\big(N\bar{N}^\ast\big)_x(x,\lambda)+\alpha^\ast(\lambda)L_n\big(N^\ast\bar{N}\big)_x(x,\lambda)\bigg)\rmd \lambda,\frac{\partial Q}{\partial c}\rangle\nonumber\\&&+\gamma^2\langle L_n\frac{\partial Q}{\partial c},\frac{\partial Q}{\partial c}\rangle=I+II+III.
\end{eqnarray}
First it is  noticed from \eqref{eq:essentia2} and the zero inner product property of the two sets \eqref{eq:set JL} and \eqref{eq:set LJ} that
\begin{eqnarray}\label{second }
II&=&2\gamma\langle \int_0^{+\infty}\bigg(\alpha(\lambda)L_n\big(N\bar{N}^\ast\big)_x(x,\lambda)+\alpha^\ast(\lambda)L_n\big(N^\ast\bar{N}\big)_x(x,\lambda)\bigg)\rmd \lambda,\frac{\partial Q}{\partial c}\rangle\nonumber\\&=&2\gamma\int_0^{+\infty}\langle \alpha(\lambda)\varrho_{n,c}(\lambda)\big(N\bar{N}^\ast\big)(x,\lambda)+\alpha^\ast(\lambda)\varrho^\ast_{n,c}(\lambda)\big(N^\ast\bar{N}\big)(x,\lambda),\frac{\partial Q}{\partial c}\rangle P(\lambda)\varrho_{n,c}(\lambda)\rmd \lambda\nonumber\\&=&0.
\end{eqnarray}
For the third term of \eqref{quadratic}, a direct computation  shows that,
\begin{eqnarray} \label{third }
 III=\gamma^2\langle (-1)^nc^{n-1}Q,\frac{\partial Q}{\partial c}\rangle=\gamma^2(-1)^nc^{n-1}\frac{\rmd H_1(Q)}{\rmd c}=\pi\gamma^2 (-1)^{n}c^{n-1}.
\end{eqnarray}
To deal with the first term in \eqref{quadratic}, using  \eqref{eq:essentia2} and \eqref{eq:product1} yields that
\begin{eqnarray} \label{first term }
&&I=\langle\int_0^{+\infty}\bigg(\alpha(\lambda)L_n\big(N\bar{N}^\ast\big)_x(x,\lambda)+\alpha^\ast(\lambda)L_n\big(N^\ast\bar{N}\big)_x(x,\lambda)\bigg)\rmd \lambda,\nonumber\\&& \int_0^{+\infty}\bigg(\alpha^\ast(\lambda)\big(N^\ast\bar{N}\big)_x(x,\lambda)+\alpha(\lambda)\big(N\bar{N}^\ast\big)_x(x,\lambda)\bigg)\rmd \lambda\rangle\nonumber\\
&=&\langle\int_0^{+\infty}\bigg(\alpha(\lambda)\varrho_{n,c}(\lambda)\big(N\bar{N}^\ast\big)(x,\lambda)+\alpha^\ast(\lambda)\varrho^\ast_{n,c}(\lambda)\big(N^\ast\bar{N}\big)(x,\lambda)\bigg)\rmd \lambda,\nonumber\\&& \int_0^{+\infty}\bigg(\alpha^\ast(\lambda)\big(N^\ast\bar{N}\big)_x(x,\lambda)+\alpha(\lambda)\big(N\bar{N}^\ast\big)_x(x,\lambda)\bigg)\rmd \lambda\rangle\nonumber\\
&=& \int_{\R_+^2}\varrho_{n,c}(\lambda)\alpha(\lambda)\alpha^\ast(\lambda')\langle\big(N\bar{N}^\ast\big)(x,\lambda),\big(N^\ast\bar{N}\big)_x(x,\lambda')\rangle\rmd \lambda\rmd \lambda'\nonumber\\
&+&\int_{\R_+^2}\varrho^\ast_{n,c}(\lambda)\alpha^\ast(\lambda)\alpha(\lambda')\langle\big(N^\ast\bar{N}\big)(x,\lambda),\big(N\bar{N}^\ast\big)_x(x,\lambda')\rangle\rmd \lambda\rmd \lambda'\nonumber\\
&=& \int_0^{+\infty}2\pi i\big(\varrho^\ast_{n,c}(\lambda)-\varrho_{n,c}(\lambda)\big)|\alpha(\lambda)|^2\rmd \lambda\nonumber\\
&=&2^{n+1}\pi\int_0^{+\infty}|\alpha(\lambda)|^2\lambda^{n+1}\big(\frac{2\lambda}{n+1}+\frac{1}{n}\big)\rmd \lambda\geq0,
\end{eqnarray}
where $I=0$ holds if and only if $\alpha(\lambda)=0$. Combining \eqref{first term }, \eqref{second } and \eqref{third }, one has
\begin{eqnarray} \label{quadratic form}
  \eqref{quadratic}=2^{n+1}\pi\int_0^{+\infty}|\alpha(\lambda)|^2\lambda^{n+1}\big(\frac{2\lambda}{n+1}+\frac{1}{n}\big)\rmd \lambda+\pi\gamma^2(-1)^{n}c^{n-1}.
\end{eqnarray}

For $z\in H^{\frac{n}{2}}_{od}(\R)$, we have $\gamma=0$, then \eqref{quadratic form} and \eqref{first term } reveal that $\langle L_{n}z,z\rangle\geq0$. Moreover, $\langle L_{n}z,z\rangle=0$ infers that $\alpha(\lambda)=0$, therefore, $z=\beta Q_x$ for $\beta\neq0$.

If $z\in H^{\frac{n}{2}}_{ev}(\R)$, we then have $\beta=0$,  In the hyperplane $\gamma=0$, $\langle L_{n}z,z\rangle\geq0$ and $\langle L_{n}z,z\rangle=0$ if and only if  $\alpha(\lambda)=0$, then one has $z=0$. Therefore, $\langle L_{n}z,z\rangle>0$ in the hyperplane $\gamma=0$ and which implies that $L_n$ can have at most one negative eigenvalue. If $n$ is odd, then  $L_n\frac{\partial Q}{\partial c}=-c^{n-1}Q<0$ and $\langle L_n\frac{\partial Q}{\partial c},\frac{\partial Q}{\partial c}\rangle=-c^{n-1}\frac{\rmd H_1(Q)}{\rmd c}=\pi (-1)^{n}c^{n-1}<0$. Therefore, $L_n$ has exactly one negative eigenvalue. If $n$ is even, then from \eqref{quadratic form} or $\langle L_n\frac{\partial Q}{\partial c},\frac{\partial Q}{\partial c}\rangle=(-1)^{n}c^{n-1}\frac{\rmd H_1(Q)}{\rmd c}=\pi c^{n-1}>0$, which means that $L_n$ has no negative eigenvalue. This completes the proof of  Lemma \ref{le3.9}.
\end{proof}

\begin{remark}\label{re:4.0}
Lemma \ref{le3.9} states that for $k\in\mathbb{N}$, the inertia of the operators $L_n$ satisfy $in(L_{2k})=(0,1)$ and $ in(L_{2k-1})=(1,1)$. One can verify, by Weyl's essential spectrum theorem, that the essential spectrum of $L_n$ ($n\geq2$) is the interval $[0,+\infty)$. It is inferred from $L_{2k}=\mathcal{R}(Q)L_{2k-1}$ or \eqref{quadratic form}  that the operator $L_{2k}$ has a positive eigenvalue $\nu=O(c^{2k})$ (with $L^2$-eigenfunctions), which may possibly be embedded into its continuous spectrum.
\end{remark}

As a direct consequence of Lemma \ref{le3.9}, one has the following spectral information of higher order linearized operators $\mathcal {T}_{n,j}:=H''_{n+2}(Q_{c_j})+(c_1+c_2)H''_{n+1}(Q_{c_j})+c_1c_2H''_n(Q_{c_j})$  (defined in $L^2(\R)$ with domain $H^{2}(\R)$) with $n\geq1$, $j=1,2$ and $c_1\leq c_2$, which are related closely to stability problem of the double solitons $U^{(2)}$. Following the same line of the proof of Lemma \ref{le3.9}, we have
\begin{corollary} \label{co3.10} For $n\geq1$ and $c_1=c_2=c$, we have  $\mathcal {T}_{n,1}=\mathcal {T}_{n,2}\geq0$, and the eigenvalue zero is double with eigenfunctions $Q_c'$ and $\frac{\partial Q_c}{\partial c}$.
For $n\geq1$ odd and $c_1<c_2$,  the operator $\mathcal {T}_{n,1}$ has one negative eigenvalue and $\mathcal {T}_{n,2}\geq0$ is positive. For $n\geq1$ even and $c_1<c_2$,  the operator $\mathcal {T}_{n,1}$ is positive and $\mathcal {T}_{n,2}\geq0$ has one negative eigenvalue.
$\mathcal {T}_{n,j}$ have zero as a simple eigenvalue with associated eigenfunctions $Q'_{c_j}$.
\end{corollary}
\begin{proof}
Similar to the proof of Lemma \ref{le3.9}, we study quadratic form related to the operator $\mathcal {T}_{n,j}$ with $z$ possessing the decomposition \eqref{decomposition of z}. One can verify that
\begin{equation}\label{direction}
\mathcal {T}_{n,j}\frac{\partial Q_{c_j}}{\partial c_j}=(c_j-c_k)(-c_j)^{n-1}Q_{c_j}, \quad\text{for}\quad k\neq j \ \text{and}\ j,k=1,2.
\end{equation}
In particular, if $c_1=c_2=c$, the function $\frac{\partial Q_{c}}{\partial c}$  belongs to the kernel of  $\mathcal {T}_{n,1}$ and $\mathcal {T}_{n,2}$. Notice that $Q'_{c}$ always belongs  to the kernel of which, therefore, zero eigenvalue is double with eigenfunctions $Q_c'$ and $\frac{\partial Q_c}{\partial c}$.  The non-negativeness of $\mathcal {T}_{n,1}$ and $\mathcal {T}_{n,2}$ follow from the same argument of Lemma \ref{le3.9}.

If $c_1<c_2$, then by \eqref{direction} and following the same line of the proof of Lemma \ref{le3.9}, the operator $\mathcal {T}_{2k+1,1}$ has a negative eigenvalue and $\mathcal {T}_{2k+1,2}\geq0$, their zero eigenvalue are simple with associated eigenfunction $Q'_{c_j}$; the operator $\mathcal {T}_{2k,1}\geq0$ and $\mathcal {T}_{2k,2}$ has a negative eigenvalue.
\end{proof}

\begin{remark}\label{re:4.3}
 The linearized operator $\mathcal{L}_2$ defined in \eqref{eq:linearized n-soliton operator} around the double solitons profile $U^{(2)}$  can be represented as follows:
 \[\mathcal{L}_2=-\frac43\partial_x^2+2HU_x+2UH\partial_x+2H(U_x\cdot)+2U\partial_x+4U^{2}+(c_1+c_2)(-H\partial_x-2U)+c_1c_2, \ U:=U^{(2)},
 \]
 which possesses the following property: the spectra $\sigma(\mathcal{L}_2)$ trends to the union of $\sigma(\mathcal {T}_{1,1})$ and $\sigma(\mathcal {T}_{1,2})$ as $t$ goes to infinity. Since from Corollary \ref{co3.10}, we know the inertia $in(\mathcal {T}_{1,1})=(1,1)$ and $in(\mathcal {T}_{1,2})=(0,1)$, then it reveals that,
 $$in(\mathcal{L}_2)=in(\mathcal {T}_{1,1})+in(\mathcal {T}_{1,2})=(1,2).$$
 In this sense, Corollary \ref{co3.10} at the case $n=1$ gives an alternative proof of Theorem 9 in \cite{LN}, which is the key spectral property in showing the orbital stability of the double solitons of the BO equation.
\end{remark}
\subsection{The spectrum of linearized operator around the BO $m$-solitons}
In order to prove Theorem \ref{thm1.1}, we need to know the spectral information of the operator $\mathcal L_m$ \eqref{eq:linearized n-soliton operator}. More precisely, the inertia of $\mathcal L_m$ called $in(\mathcal L_m)$ has to be determined. The aim of this subsection is to show  the following result.
\begin{lemma}\label{lemma3.6}
The operator $\mathcal L_m$ defined in $L^2(\R)$ with domain $H^{\frac{m}{2}}(\R)$ verifies the following spectral property
\begin{equation}\label{inertia L}
in(\mathcal L_m)=\big(n(\mathcal L_m),z(\mathcal L_m)\big)=\big([\frac{m+1}{2}],m\big).
\end{equation}
\end{lemma}
To this aim, for $j=1,2...,m$, recall that $L_{m,j}=S_m''(Q_{c_j})$ is defined in \eqref{formula of L}. The spectrum of $\mathcal L_m$ tends to the unions of $L_{m,j}$, that is $\sigma(\mathcal L_m)\rightarrow\bigcup_{ j=1}^m\sigma(L_{m,j})$ as $t\rightarrow+\infty$. The result \eqref{inertia L} follows directly from the following statement which concerning the inertia of the operators $L_{m,j}$, $j=1,2,\cdot\cdot\cdot,m$.

\begin{proposition} \label{pr3.11}  (1). $L_{m,2k-1}$  (defined in $L^2(\R)$ with domain $H^{m}(\R)$) has zero as a simple eigenvalue and exactly one negative eigenvalue for  $1\leq k\leq[\frac{m+1}{2}]$, i.e, $in(L_{m,2k-1})=(1,1)$;
\ \ \ (2). $L_{m,2k}$  (defined in $L^2(\R)$ with domain $H^{m}(\R)$) has zero as a simple eigenvalue and no negative eigenvalues for $1\leq k\leq[\frac{m}{2}]$, i.e, $in(L_{m,2k})=(0,1)$.
\end{proposition}

\begin{proof} The proof follows the same line of the proof of Lemma \ref{le3.9}.
We consider the operator $L_{m,j}=S_{m}''(Q_{c_j})$ for $1\leq j \leq m$ and compute the quadratic form $\langle L_{m,j}z,z\rangle$ under a special decomposition of $z$ \eqref{decomposition of z}. Recall from \eqref{formula of L} that the form of $L_{m,j}$ which is a combination of the operators $H_{n+1}''(Q_{c_j})+c_jH_{n}''(Q_{c_j})$, and those $\sigma_{j,k}>0$ are the elementally symmetric functions of $c_1,c_2,\cdot\cdot\cdot,c_{j-1},c_{j+1},\cdot\cdot\cdot,c_m$. Moreover, one has
 \begin{equation}\label{eigenvalue}
 L_{m,j}\frac{\partial Q_{c_j}}{\partial c_j}=-\prod_{k\neq j}^m(c_k-c_j)Q_{c_j}:=\Gamma_jQ_{c_j}.
\end{equation}
The quadratic form $\langle L_{m,j}z,z\rangle$ (for $z\in H^{\frac{m}{2}}(\R)$) can be evaluated similar to \eqref{quadratic} as follows
 \begin{eqnarray*}
&&\langle L_{m,j}z,z\rangle=\langle\int_0^{+\infty}\bigg(\alpha(\lambda)L_{m,j}\big(N\bar{N}^\ast\big)_x(x,\lambda)+\alpha^\ast(\lambda)L_{m,j}\big(N^\ast\bar{N}\big)_x(x,\lambda)\bigg)\rmd \lambda,\nonumber\\&& \int_0^{+\infty}\bigg(\alpha(\lambda)\big(N\bar{N}^\ast\big)_x(x,\lambda)+\alpha^\ast(\lambda)\big(N^\ast\bar{N}\big)_x(x,\lambda)\bigg)^\ast\rmd \lambda\rangle\nonumber\\&&+2\gamma\langle \int_0^{+\infty}\bigg(\alpha(\lambda)L_{m,j}\big(N\bar{N}^\ast\big)_x(x,\lambda)+\alpha^\ast(\lambda)L_{m,j}\big(N^\ast\bar{N}\big)_x(x,\lambda)\bigg)\rmd \lambda,\frac{\partial Q_{c_j}}{\partial c_j}\rangle\nonumber\\&&+\gamma^2\langle L_{m,j}\frac{\partial Q_{c_j}}{\partial c_j},\frac{\partial Q_{c_j}}{\partial c_j}\rangle=\sum_{n=1}^m\bigg(2^{n+1}\pi\sigma_{j,m-n}\int_0^{+\infty}|\alpha(\lambda)|^2\lambda^{n+1}\big(\frac{2\lambda}{n+1}+\frac{1}{n}\big)\rmd \lambda\bigg)+\pi\gamma^2\Gamma_j.
 \end{eqnarray*}
One can check that the symbol of the principle part of $ L_{m,j}$ evaluated at $\lambda$ is
\begin{eqnarray}\label{symbol2}
\widehat{S_m''(0)}(\lambda)&=&\sum_{n=1}^m\sigma_{j,m-n}\rho_{n,c_j}(\lambda)>0.
\end{eqnarray}
Then the first term of the quadratic form $\langle L_{m,j}z,z\rangle$ is nonnegative and equals to zero if and only if $\alpha(\lambda)=0$.

If $j$ is even, then in view of  the definition of $\Gamma_j$ \eqref{eigenvalue}, one has $\Gamma_j>0$ and $\langle L_{m,j}z,z\rangle\geq0$ and $\langle L_{m,j}z,z\rangle=0$ if and only if $\alpha(\lambda)=0$ and $\gamma=0$, which indicates that $z=\beta Q'_{c_j}$. Hence $L_{m,j}\geq0$ and zero is simple with associated eigenfunction $Q'_{c_j}$.

If $j$ is odd, then one has $\Gamma_j<0$, we investigate $z$ in $H^{\frac{m}{2}}_{ev}(\R)$ and $H^{\frac{m}{2}}_{od}(\R)$, respectively. If $z\in H^{\frac{m}{2}}_{od}(\R)$, then $\gamma=0$. Then one has $\langle L_{m,j} z,z\rangle\geq0$ and $\langle L_{m,j}z,z\rangle=0$ if and only if  $\alpha(\lambda)=0$. Then $z=\beta Q'_{c_j}$ with $\beta\neq0$, which indicates that zero is simple with associated eigenfunction $Q'_{c_j}$.

If $z\in H^{\frac{m}{2}}_{ev}(\R)$, then $\beta=0$. In the hyperplane $\gamma=0$, $\langle L_{m,j}z,z\rangle\geq0$ and $\langle L_{m,j}z,z\rangle=0$ if and only if  $\alpha(\lambda)$. Therefore, $\langle L_{m,j}z,z\rangle>0$ in the hyperplane $\gamma=0$ and which implies that $L_{m,j}$ can have at most one negative eigenvalue. Since $L_{m,j}\frac{\partial Q_{c_j}}{\partial c_j}=\Gamma_jQ_{c_j}<0$ and $$\langle L_{m,j}\frac{\partial Q_{c_j}}{\partial c_j},\frac{\partial Q_{c_j}}{\partial c_j}\rangle=\Gamma_j\frac{\rmd H_1(Q_{c_j})}{\rmd c_j}<0.$$ Therefore, $L_{m,j}$ has exactly one negative eigenvalue.  This implies the desired result as advertised in the statement of Proposition \ref{pr3.11}.
\end{proof}

\begin{proof}[Proof of Lemma \ref{lemma3.6}]
From the invariance of inertia of $\mathcal L_m$, we know that
\[
in(\mathcal L_m)=\big(n(\mathcal L_m),z(\mathcal L_m)\big)=\sum_{j=1}^min(\mathcal L_{m,j})=\big([\frac{m+1}{2}],m\big).
\]
The proof is concluded.
\end{proof}

\begin{remark}\label{re:4.6}
In view of \eqref{operator identity3} and \eqref{operator identity4}, one may also investigate the spectrum of the operator $\mathcal{J}\mathcal L_m$ to show the spectral stability of the BO $m$-solitons and then the spectrum of the operator $\mathcal L_m$. The idea is similar to the $m=1$ case, by employing the eigenvalue problem \eqref{space lax}, we can derive the eigenvalues and the associated eigenfunctions of the recursion operator around the $m$-solitons profile $U^{ (m)}(x)$. Then we need to show the eigenfunctions plus their derivatives with respect to the eigenvalues $\lambda_j$ ($j=1,2,\cdot\cdot\cdot,m$) form a basis in $L^2(\R)$. Finally, by a direct verification of the quadratic form $\langle\mathcal L_mz,z\rangle$ (with function $z$ decomposes upon the above bases), one can also derive the inertia of the operator $\mathcal L_m$. In fact, we can show the following
\begin{equation*}
n(\mathcal L_m)=-\sum_{1\leq j=2k-1\leq m} \operatorname{sgn}\bigg(\langle\mathcal L_m\frac{\partial U^{ (m)}}{\partial c_j},\frac{\partial U^{ (m)}}{\partial c_j}\rangle\bigg)=[\frac{m+1}{2}],\ k=1,2,\cdot\cdot\cdot,[\frac{m+1}{2}]
\end{equation*}
which reveals that the negative eigenvalues of  $\mathcal L_m$ are generated by the directions $\frac{\partial U^{ (m)}}{\partial c_j}$ for odd $j=1,3,\cdot\cdot\cdot,2[\frac{m+1}{2}]-1$.
\end{remark}
\section{Proof of the main results} \label{sec_5}
\setcounter{equation}{0}
This section is devoted to the proof of Theorem ~\ref{thm1.1} and Theorem ~\ref{thm1.3}. To do this, we need to prove that multi-solitons of \eqref{eq: BO} verify a stability criterion established by Maddocks and Sachs \cite{MS93}.  Recall that the variational principle \eqref{E-L}
 is the gradient of the functional \eqref{Lyafun} evaluated at $u=U^{ (m)}$. In general, the $m$-solitons  $U^{ (m)}(t,x)$ is not a minimum of $S_m$, rather, it is at best a constrained and nonisolated minimum of the following minimization problem
\begin{eqnarray*}
\min H_{m+1}(u(t)) \quad\quad \text{subject to} \quad  H_{j}(u(t))= H_{j}(U^{ (m)}(t)), \quad j=1,2,...,m.
\end{eqnarray*}
Now, we consider the second variation self-adjoint  operator $\mathcal L_m(t)$ defined by \eqref{eq:linearized n-soliton operator} and denote by
$$n(\mathcal L_m(t))$$ the number of negative eigenvalue of $\mathcal L_m(t)$. Observe that the above defined objects are  {\it a priorily} time-dependent. We also define the $m\times m$ Hessian matrix by
\begin{eqnarray}\label{3.31}
D(t):=\big\{\frac{\partial^2S_m(U^{ (m)}(t))}{\partial \mu_i\partial \mu_j}\big\},
\end{eqnarray}
and denote by $$p(D(t))$$ the number of positive eigenvalue of $D(t)$. Since $S_m(t)$ is a conserved quantity for the flow of \eqref{eq: BO}, the matrix $D(t)$ is independent of $t$. The proof of Theorem \ref{thm1.1} relies on the following theoretical result, which was first stated by Maddocks and Sachs \cite[Lemma 2.1]{MS93}. Maddocks-Sachs \cite{MS93} provided an outline for the proof of this result. For reader's convenience, we give a detailed proof here
\begin{proposition}\label{pr2.1}
Suppose that
\begin{equation}\label{n=p}
n(\mathcal L_m)=p(D).
\end{equation}
Then there exists a constant $C>0$ such that $U^{ (m)}$ is a non-degenerate unconstrained
minimum of the augmented Lagrangian (Lyapunov functional)
\begin{equation}\label{eq:augmented Lagrangian}
\Delta(u):=S_m(u)+\frac{C}{2}\sum_{j=1}^m\big(H_j(u)-H_{j}(U^{ (m)})\big)^2.
\end{equation}
As a consequence, $U^{ (m)}(t,x)$ is dynamically stable.
\end{proposition}

\begin{proof}
Since the functional $S_m$ depends only on wave speeds ${\mathbf c}$ and not on $t$
or ${\mathbf x}$. Hence, by construction of the augmented Lagrangian $\Delta$, any $m$-solitons with parameters ${\mathbf c}$ is a critical point of $\Delta$.
Moreover, there exists $\gamma>0$ (which, as well as $C$, can be chosen independently of ${\mathbf x}$) such that
for any $U^{ (m)}(\cdot,\cdot;{\mathbf c},{\mathbf x})$ and for any $h \in H^{\frac{m}{2}}(\R)$ such that
\[
\langle\nabla_{{\mathbf x}}U^{ (m)}(t,\cdot;{\mathbf c},{\mathbf x}),h\rangle=0,
\]
one has
\[
\langle \Delta''(U^{ (m)}(t,\cdot;{\mathbf c},{\mathbf x}))h,h\rangle\geq \gamma \|h\|^2_{H^{\frac{m}{2}}}.
\]
Now for any $\in H^{\frac{m}{2}}(\R)$ such that
\[\inf_{{\mathbf y}\in \R^m}\|u-U^{ (m)}(t,\cdot;{\mathbf c},{\mathbf y})\|_{H^{\frac{m}{2}}}<\varepsilon,\]
there exists ${\mathbf y}_u\in \R^m$ such that
\begin{eqnarray*}
&&\inf_{{\mathbf y}\in \R^m}\|u-U^{ (m)}(t,\cdot;{\mathbf c},{\mathbf y})\|^2_{H^{\frac{m}{2}}}\leq \frac{2}{\gamma}\bigg(\Delta(u)-\Delta(U^{ (m)}(t,\cdot;{\mathbf c},{\mathbf y}_u))\bigg)\\&&=\frac{2}{\gamma}\bigg(\Delta(u)-\Delta(U^{ (m)}(t,\cdot;{\mathbf c},{\mathbf x}))\bigg)=\frac{2}{\gamma}\bigg(\Delta(u_0)-\Delta(U^{ (m)}(0,\cdot;{\mathbf c},{\mathbf x}))\bigg)\\&&\leq C\|u_0-\Delta(U^{ (m)}(0,\cdot;{\mathbf c},{\mathbf x}))\|^2_{H^{\frac{m}{2}}}\leq C\delta^2<\varepsilon.
\end{eqnarray*}
Here we used the conservation of the augmented Lagrangian $\Delta$ by the \eqref{eq: BO} flow, given an initial data $u_0$ sufficiently close to an $m$-solitons profile $U^{ (m)}(0,\cdot;{\mathbf c},{\mathbf x})$, the closeness to the $m$-solitons manifold with speeds ${\mathbf c}$ is preserved for all time.
\end{proof}
Therefore, to complete the proof of Theorem~\ref{thm1.1}, it is sufficient to verify~\eqref{n=p}. We start with the count of the number of positive eigenvalues of the Hessian matrix $D$, which has been shown in \cite{M06}.

\begin{lemma}
 \label{lem:2.2.}
  For all
${\mathbf{c}}=(c_1,\ldots,c_m),{\mathbf{x}}=(x_1,\ldots,x_m)$ with $0<c_1 < \cdots<c_m$, we have
\[
  p(D)= [\frac{m+1}{2}].
 \]
\end{lemma}

\begin{proof}
The Hessian matrix $D$ is defined by \eqref{3.31}. It is a real symmetric matrix,
whose elements can be calculated explicitly for the $m$-solitons.
Indeed, since  $m$-solitons are reflectionless potentials, one takes $\beta=0$ in \eqref{trace for H},  the $n$-th conservation law  corresponding to $u=U^{ (m)}$ reduces to
$$H_n(U^{ (m)})=\pi (-1)^{n+1}\sum_{l=1}^m\frac{c_l^{n}}{n}.$$
If we regard $S_m$ as a function of $\mu_j$ $j=1, 2, ..., m)$, then from \eqref{Lyafun} and \eqref{E-L}, one has
 $$\frac{\partial S_m}{\partial \mu_j}=H_{j}, \quad j=1, 2, ..., m. $$
Hence the elements of the matrix $D$ are as follows
\begin{equation}\label{B3}
d_{jk}:=\frac{\partial H_{j}}{\partial \mu_k}
=\pi (-1)^{j+1}\sum_{l=1}^mc_l^{j-1}\frac{\partial c_l}{\partial \mu_k}.
\end{equation}
Let $A=(a_{jk})_{1\leq j,k\leq m}$ and $B=(b_{jk})_{1\leq j,k\leq m}$ be $m\times m$ matrices with
elements
$$a_{jk}=\pi (-1)^{j+1}c_k^{j-1}, $$
\begin{equation}\label{B4b}
b_{jk}=\frac{\partial \mu_j}{ \partial c_k},
\end{equation}
respectively. From \eqref{B4b} and the fact that $c_j\not=c_k$ for $j\not=k$, we see that
$$\det B=\prod_{1\leq j<k\leq m}(c_k-c_j)\neq0.$$
Thus $B$ is invertable. Now we can rewrite \eqref{B3} in the form
\begin{equation}\label{B5}
D=AB^{-1},
\end{equation}
which implies that $B^TDB=B^TA.$ From the Sylvester's law of inertia, one deduces that the number of
positive eigenvalues of $D$ coincides with that of $B^TA$. We know that $B^TA$ is a diagonal matrix
since the $(j, k)$ element of $B^TA$ becomes
\begin{equation}\label{B8}
(B^TA)_{jk}=\pi\sum_{l=1}^m(-1)^{l+1}\frac{\partial\sigma_{m-l+1}}{\partial c_j}c_k^{l-1}=\delta_{jk}\prod_{l\neq j}(c_l-c_k).
\end{equation}
It is  easy to see that the number of positive eigenvalues of $B^TA$ is equal to
$\left[\frac{m+1}{2}\right]$, which concludes the proof.
 \end{proof}

\begin{proof} [Proof of Theorem \ref{thm1.1}]
By Lemma \ref{lem:2.2.} and Lemma \ref{lemma3.6}, one has that $n(\mathcal L_m)=p(D)=[\frac{m+1}{2}]$. The proof of Theorem \ref{thm1.1} is obtained directly in view of Proposition \ref{pr2.1}, since $U^{ (m)}(t,x)$
is now an (non-isolated) unconstrained minimizers of the augmented Lagrangian \eqref{eq:augmented Lagrangian}
which therefore serves as a Lyapunov function.
\end{proof}

Now we remain to prove  Theorem \ref{thm1.3}.
\begin{proof}[Proof of Theorem \ref{thm1.3}]
The linearized operators around the $m$-solitons $\mathcal {L}_m=S''_m(U^{ (m)})$ possess $[\frac{m+1}{2}]$ negative eigenvalues, which has been verified from \eqref{inertia L}. Next, we need to prove \eqref{eigenvalue1.3}. As  $t$ goes to $\infty$, the spectrum $\sigma(\mathcal {L}_m(t))$ of $\mathcal {L}_m(t)$ converges
to the union of the spectrum $\sigma(\mathcal {L}_{m,j})$ of $\mathcal {L}_{m,j}=S_m''(Q_{c_j})$, namely
\[
\sigma(\mathcal {L}_m(t))\rightarrow \bigcup_{j=1}^m\sigma(\mathcal {L}_{m,j}), \quad \text{as} \quad t\rightarrow +\infty.
\]
Since for each $m$, the operators $\mathcal {L}_m(t)$ are isoinertial, the spectrum of which $\sigma(\mathcal {L}_m(t))$ is independent of $t$. Therefore, the negative eigenvalues of $\mathcal {L}_m$ are exactly the same with the negative eigenvalues of $\mathcal {L}_{m,j}$ for all $j=1,2,\cdot\cdot\cdot,m$. In view of Lemma \ref{lemma3.6}, $\mathcal {L}_{m,j}$ possesses negative eigenvalues if $j=2k-1$ and $1\leq k\leq [\frac{m+1}{2}]$. We will show that such negative eigenvalues are exactly $\nu_k$ \eqref{eigenvalue1.3}. Indeed, by induction, $m=1$ is verified in \eqref{ pos eigen}, the associated negative eigenvalue is $\nu_1=-\frac{\sqrt{5}+1}{2}c$ \eqref{ pos eigen}. Suppose now \eqref{eigenvalue1.3} holds for $m=K$, namely, the $[\frac{K+1}{2}]$-th negative eigenvalue of $\mathcal {L}_{K}$ is
\begin{eqnarray}\label{Keigenvalue}
\nu^K_k:=-Cc_{2k-1}\prod_{j\neq 2k-1}^K(c_j-c_{2k-1}),\quad k=1,2,\cdot\cdot\cdot,[\frac{K+1}{2}].
\end{eqnarray}

If $m=K+1$ even, in this case $[\frac{K+1}{2}]=[\frac{K+2}{2}]$, for $k=1,2,\cdot\cdot\cdot,[\frac{K+1}{2}]$, one has
\begin{equation}\label{Kplus1 and K}
\mathcal {L}_{K+1,2k-1}=S_{K+1}''(Q_{c_{2k-1}})=\left(\mathcal{R}(Q_{c_{2k-1}})+c_{K+1}\right)I_{K}''(Q_{c_{2k-1}}).
\end{equation}
By Lemma \ref{le3.6}, the operator $(\mathcal{R}(Q_{c_{2k-1}})+c_{K+1})$ has an eigenvalue $c_{K+1}-c_{2k-1}>0$, the continuous spectrum is $[c_{K+1},+\infty)$ whose generalized eigenfunctions are not in $L^2(\R)$. Therefore, the $[\frac{K+2}{2}]$-th
negative eigenvalues of $\mathcal {L}_{K+1,2k-1}$ are
\begin{eqnarray}\label{K+1eigenvalue}
\nu^{K+1}_k:&=&\big(c_{K+1}-c_{2k-1}\big)\nu^K_k=-C\big(c_{K+1}-c_{2k-1}\big)c_{2k-1}\prod_{j\neq 2k-1}^K(c_j-c_{2k-1})\nonumber\\&=&-Cc_{2k-1}\prod_{j\neq 2k-1}^{K+1}(c_j-c_{2k-1}),\quad k=1,2,\cdot\cdot\cdot,[\frac{K+1}{2}],
\end{eqnarray}
where the constant $C>0$ is different with respect to \eqref{Keigenvalue}.

If $m=K+1$ odd, in this case $[\frac{K+1}{2}]+1=[\frac{K+2}{2}]$. For $k=1,2,\cdot\cdot\cdot,[\frac{K+1}{2}]$, following by the same argument, the front $[\frac{K+1}{2}]$
negative eigenvalues of $\mathcal {L}_{K+1}$ are given by \eqref{K+1eigenvalue}. Now we compute the last negative eigenvalue which has been proven Lemma \ref{lemma3.6}. Since
\begin{equation}\label{Kplus1 and K2}
\mathcal {L}_{K+1,K+1}=S_{K+1}''(Q_{c_{K+1}})=\left(\mathcal{R}(Q_{c_{K+1}})+c_{j}\right)\tilde{S}_{K}''(Q_{c_{K+1}}),
\end{equation}
where $\tilde{S}_{K}$ is the action that with a wave speed $c_j$ in $S_K$ replacing to $c_{K+1}$ for some $1\leq j\leq K$.
By the assumption in \eqref{Keigenvalue}, the discrete eigenvalue of the operator $\tilde{S}_{K}''(Q_{c_{K+1}})$ is
\[
-Cc_{K+1}\prod_{l\neq j}^K(c_l-c_{K+1}).
\]
Since by Lemma \ref{le3.6}, the operator $\mathcal{R}(Q_{c_{K+1}})+c_{j}$ has an eigenvalue $c_j-c_{K+1}<0$, the continuous spectrum of which is the interval $[c_{j},+\infty)$ and the generalized eigenfunctions are not in $L^2(\R)$. Therefore, the last
negative eigenvalues of $\mathcal {L}_{K+1}$ is
\begin{equation}\label{K+1eigen}
\nu^{K+1}_{[\frac{K+2}{2}]}:=\big(c_j-c_{K+1}\big)\big(-Cc_{K+1}\prod_{l\neq j}^K(c_l-c_{K+1})\big)=-Cc_{K+1}\prod_{l=1}^K(c_l-c_{K+1}).
\end{equation}
 The proof of Theorem \ref{thm1.3} is concluded by combining \eqref{K+1eigenvalue} and \eqref{K+1eigen}.
\end{proof}
\section*{Acknowledgment}
 Y. Lan acknowledges the support of the China National Natural Science Foundation under grant number 12201340, Z. Wang acknowledges the support of the Ministry of Education under grant number HZKY20220105 and Guangdong Natural Science Foundation under grant number 2023A1515010706. Z. Wang is also indebted to Prof. Stefan Le Coz for stimulating discussions.

  \section*{Data Availability}
 The data that supports the findings of this study are available within the article.

 \section*{Conflict of interest}
 The authors have no conflicts to disclose.

\end{document}